\RequirePackage{fix-cm}

\documentclass[smallextended,envcountsect]{svjour3}       
\smartqed  

\makeatletter
\def\cl@chapter{\@elt {theorem}}
\makeatother
\usepackage{graphicx}
\usepackage{amsmath}
\usepackage[utf8]{inputenc} 
\usepackage[T1]{fontenc}    
\usepackage{hyperref}       
\usepackage{cleveref}
\usepackage{url}            
\usepackage{booktabs}       

\usepackage{nicefrac}       
\usepackage{microtype}      
\usepackage{xcolor}         
\usepackage{enumitem}
\usepackage{indentfirst,latexsym,bm}
\usepackage{diagbox}
\usepackage{pifont}
\usepackage{comment}

\usepackage{nicematrix}
\definecolor{lavender}{rgb}{0.9, 0.9, 0.98}

\usepackage{amsfonts}
\usepackage{mathrsfs}
\usepackage{amsthm}
\usepackage{algorithm}
\usepackage{color}
\usepackage{algorithmicx}
\usepackage{algpseudocode}
\usepackage{graphicx}
\usepackage{extarrows}
\usepackage{tikz}
\usepackage{diagbox}
\usepackage{pifont}
\usepackage{multirow}
\graphicspath{{figure/}}
 
\usepackage{amsmath}

\DeclareMathOperator*{\argmin}{argmin}

\colorlet{color1}{blue}
\colorlet{color2}{red!50!black}

\definecolor{ivory}{RGB}{218,215,203}

\definecolor{cuhkp}{RGB}{98,56,105} 	
\definecolor{cuhkpl}{RGB}{152,24,147} 	
\definecolor{cuhkb}{RGB}{219,160,1} 	
\definecolor{cuhkbd}{RGB}{178,129,0} 	
\definecolor{cuhkr}{RGB}{88,35,155}  	

\definecolor{blackp}{RGB}{0,0,0} 
\definecolor{redp}{RGB}{255,0,0}
\definecolor{orangep}{RGB}{255,128,0}
\definecolor{brownp}{RGB}{128,77,0}
\definecolor{yellowp}{RGB}{255,230,0}
\definecolor{greenp}{RGB}{128,230,0}
\definecolor{bluep}{RGB}{0,128,255}
\definecolor{purplep}{RGB}{152,24,147}
\definecolor{pinkp}{RGB}{230,0,128}        

\usepackage{hyperref}[6.83]
\hypersetup{
  colorlinks=true,
  frenchlinks=false,
  pdfborder={0 0 0},
  naturalnames=false,
  hypertexnames=false,
  breaklinks,
  allcolors = bluep,
  urlcolor = color1,
}

\usepackage{mdframed}

\mdfdefinestyle{assumptionbox}{
backgroundcolor=white,
    linecolor=black, 
    linewidth=.5pt,
    topline=true,
    bottomline=true,
    rightline=true,
    leftmargin=0pt,
    innertopmargin=6pt,
    innerbottommargin=6pt,
    innerrightmargin=8pt,
    innerleftmargin=8pt,
    shadow=false,
    roundcorner=0pt,
}

\mdfdefinestyle{proofbox}{
    backgroundcolor=LightBlue!10,
    linecolor=ivory, 
    linewidth=8pt,
    topline=false,
    bottomline=false,
    rightline=false,
    leftline=true,
    leftmargin=0pt,
    innertopmargin=8pt,
    innerbottommargin=8pt,
    innerrightmargin=10pt,
    innerleftmargin=10pt,
    shadow=false,
    shadowcolor=gray!10,
    shadowsize=2pt,
    roundcorner=0pt,
}

\theoremstyle{thmstyleone}%

\newtheorem{thm}{Theorem}[section]
\newtheorem{lem}[thm]{Lemma}

\newtheorem{rem}[thm]{Remark}

\newtheorem{defn}[thm]{Definition}

\usepackage{subfigure}
\usepackage{amssymb}

\newcommand{\bR}{\mathbb{R}}
\newcommand*\vnorm[1]{\left\| #1\right\|}

\renewcommand {\AA}  { {A} }

\newcommand{\MR}{\mathrm{MINRES}}
\newcommand{\NPC}{\mathrm{NPC}}

\newcommand {\xx}  { { p} }

\newcommand {\rr}  { { r} }

\newcommand {\pp}  { { z} }

\newcommand {\vv}  { { v} }

\newcommand {\bb}  { { b} }

\newcommand {\dd}  { {d} }

\newcommand {\zero}  { {\bf 0} }

\newcommand{\iprod}[2]{\langle #1, #2 \rangle}
\newcommand{\red}[1]{\textcolor{red}{#1}}

\newcommand{\bN}{\mathbb{N}}

\newcommand{\wrt}{\text{w$.$r$.$t$.$}}
\newcommand {\dist}  { {\mathrm{dist}} }

\newcommand {\cA}  { {\mathsf{A}} }
\newcommand {\cB}  { {\mathsf{B}} }

\newcommand {\cZ}  { {\mathsf{Z}} }
	
\definecolor{MyRed}{RGB}{224,60,138} 	
\definecolor{DeepGreen}{RGB}{92,172,129} 	
\definecolor{DeepRed}{RGB}{167,67,67}  

\newcommand{\cmark}{{\color{DeepGreen}\ding{51}}}
\newcommand{\xmark}{{\color{DeepRed}\ding{55}}}    

\begin{document}

\title{A MINRES-based Linesearch Algorithm for Non- convex Optimization with Non-positive Curvature Detection
}

\titlerunning{MINRES-based Linesearch Algorithm with Non-positive Curvature Detection}        

\author{Hanfeng Zeng \and Yang Liu \and  Wenqing Ouyang \and Andre Milzarek
}

\authorrunning{Zeng, Liu, Ouyang, and Milzarek} 

\institute{
  Hanfeng Zeng  \& Andre Milzarek \at
  School of Data Science, The Chinese University of Hong Kong, Shenzhen, (CUHK-Shenzhen), Guangdong, 518172, P.R. China. 
  \and
  Wenqing Ouyang \at
  IEOR, Columbia University, New York, USA.
  (\url{wo2205@columbia.edu})  
  \and
  Yang Liu \at
  Mathematical Institute, University of Oxford, UK.
  (\url{yang.liu@maths.ox.ac.uk})           
  \and
  Corresponding author: Andre Milzarek (\url{andremilzarek@cuhk.edu.cn})
}

\date{Received: date / Accepted: date}

\maketitle

\begin{abstract}
We propose a MINRES-based Newton-type algorithm for solving unconstrained nonconvex optimization problems. Our approach uses the minimal residual method (MINRES), a well-known solver for indefinite symmetric linear systems, to compute descent directions that leverage second-order and non-positive curvature (NPC) information. Comprehensive asymptotic convergence properties are derived under standard assumptions. In particular, under the Kurdyka-{\L}ojasiewicz inequality and a mild NPC-detectability condition, we prove that our algorithm can avoid strict saddle points and converge to second-order critical points. This is primarily achieved by integrating proper regularization techniques and forward linesearch mechanisms along NPC directions. Furthermore, fast local superlinear convergence to potentially non-isolated minima is established, when the local Polyak-{\L}ojasiewicz condition is satisfied. Numerical experiments on the CUTEst test collection and on a deep auto-encoder problem illustrate the efficiency of the proposed method.


\end{abstract}

\section{Introduction} \label{intro}
In this work, we consider the unconstrained optimization problem:
\begin{equation}
    \label{tar-func}
    \min_{x\in\mathbb{R}^n}~f(x),
\end{equation}
where $f : \bR^n \to \bR$ is twice continuously differentiable but not necessarily convex. Nonconvexity is a ubiquitous and important characteristic of many modern applications including, e.g., machine and deep learning \cite{bottou2018optimization}, statistics and regression \cite{draper1998applied}, and nonconvex M-estimators \cite{loh2013regularized}. Nonconvex problems are known to exhibit several different types of critical points. Among these, saddle points, i.e., critical points at which the Hessian contains non-positive curvature ($\NPC$) information, are typically undesirable, whereas local minima represent the intended target of optimization. Truncated Newton methods, equipped with suitable inner linear system solvers, constitute a classical approach to solve problem \eqref{tar-func}. Here, embedded linear system solvers allow handling and exploiting the second-order information of the underlying problem in an effective way. As a result, such methods not only identify critical points, but typically also have the capacity to escape saddle points. In this paper, we consider a linesearch-based truncated Newton method, where the descent direction $d$ is computed as the approximate solution of the system
\begin{equation} \label{sub: LS}
    Bd=-g, \quad g= \nabla f(x), \quad B \approx \nabla^2 f(x).
\end{equation}
Common choices for inner solvers to compute such $d$ include the conjugate gradient method (CG) \cite{hestenes1952methods}, the conjugate residual method (CR) \cite{stiefel1955relaxationsmethoden}, and the minimum residual method ($\MR$ or MR) \cite{paige1975solution}, among others. 

Recently, Liu and Roosta \cite{liu2022minres} have shown that $\MR$ can effectively detect $\NPC$ information and generate $\NPC$ directions at a low computational cost. Motivated by this favorable property, we develop a second-order-type algorithm that uses $\MR$ to solve a regularized version of \eqref{sub: LS}:
\begin{equation}\label{eq:subreg}
    \bar Bd=-g, \quad \text{where} \quad \bar{B} := B + \zeta I,
\end{equation}
and $\zeta $ is the regularization parameter. We establish various desirable, asymptotic properties of the proposed method including iterate convergence guarantees under the Kurdyka-{\L}ojasiewicz (KL) inequality, avoidance of strict saddle points, and superlinear convergence to non-isolated local minima. \\[2mm]
\noindent\textit{Related Work.} Truncated Newton methods are fundamental approaches that leverage second-order information at a relatively low cost. Classically, the conjugate gradient method (CG) \cite{hestenes1952methods} has been used as the inner solver in such algorithms to obtain inexact and cheaper solutions to Newton's equation, see  \cite{dembo1983truncated,royer2020newton}. In \cite{dahito2019conjugate}, Dahito and Orban use CR as an alternative to CG and show comparable performance. Furthermore, the authors in \cite{liu2021convergence,roosta2022newton} consider variants based on $\MR$, a more general form of CR where $B$ can be indefinite, but the analysis is limited to invex problems. In \cite{liu2022newton}, Liu and Roosta discuss non-asymptotic complexity bounds for such Newton-MR methods for general nonconvex problems. More recently, in \cite{lim2025complexity}, Lim and Roosta extend the overall setting and complexity bounds derived in \cite{liu2022newton} to Newton-type-MR methods with inexact Hessian information. To our knowledge, asymptotic convergence results for truncated $\MR$-based Newton methods currently do not seem to be available in the literature in the general nonconvex setting. 

Regularization represents another prominent strategy for globalizing and robustifying Newton-type methods. As illustrated in \eqref{eq:subreg}, this involves the introduction of a regularization parameter $\zeta_k$ and the modification of the Newton update to $x_{k+1} = x_k - (\nabla^2 f(x_k) + \zeta_k I )^{-1} \nabla f(x_k)$. Classical results demonstrate that regularized Newton approaches preserve quadratic convergence under the condition $ \zeta_k \le C\|\nabla f(x_k)\| $, $C > 0$, \cite{li2001modified}. In the seminal work \cite{nesterov2006cubic}, Nesterov and Polyak  introduce the cubic regularization method, where $ \zeta_k = C\|x_{k+1} - x_k\|$, and establish global worst-case complexity guarantees. Recently, Mishchenko \cite{mishchenko2023regularized} proposes a regularization scheme that combines elements of cubic regularization with the Levenberg-Marquardt (LM) method \cite{marquardt1963algorithm}. In this approach, $\zeta_k$ is chosen as $\zeta_k = \sqrt{C \| \nabla f(x_k) \|}$, achieving the fast global rate of cubic regularization methods while requiring only a single matrix inversion per iteration. Truncation and regularization techniques are often integrated to enhance both theoretical guarantees and practical efficiency, see, e.g., \cite{fasano2006truncated,li2009truncated}.

Saddle point avoidance results for zero- and first-order methods typically rely on the stable manifold theorem and ensure that these approaches can almost always avoid strict saddle points, asymptotically, with proper stepsizes, \cite{lee2019first,panageas2019first,schaeffer2020extending,o2019behavior,vlatakis2019efficiently}. For second-order methods, escape mechanisms commonly use directions $v$ that indicate non-positive curvature of the Hessian $B=\nabla^2 f(x)$, i.e., $\iprod{v}{Bv} \le 0$. Such directions can help avoid saddle points theoretically and improve the numerical performance,  \cite{goldfarb1980curvilinear,gould2000exploiting,more1979use,goldfarb2017using,curtis2019exploiting}. Detecting non-positive curvature is fairly straightforward, as most inner solvers already calculate the terms $\iprod{v}{Bv}$ as a byproduct (for some $v\in\bR^n$). In practice, the direction $v$ is usually chosen as the eigenvector related to the leftmost eigenvalue of $B$, i.e., $\iprod{v}{Bv} \le \lambda_{\min}(B)\|v\|^2 <0$, \cite{gould2000exploiting,curtis2019exploiting}, or through eigenpair-related approaches, \cite{paternain2019newton}. Eigenpair-related directions can be costly to compute, e.g., by applying the Lanczos process which requires the storage of additional matrices or to rerun certain recurrences, \cite{conn2000trust,gould2000exploiting}. By contrast, in a non-asymptotic setting, Royer et al. \cite{royer2020newton} introduce the damped Newton-CG method, which employs capped-CG and a Minimum Eigenvalue Oracle (MEO) with randomly generated vectors when the norm of the gradient is sufficiently small. By doing so, NPC directions can be constructed in an eigenpair-free manner and probabilistic results are established for avoiding saddles. A similar MEO-based approach for trust region methods is presented in \cite{curtis2021trust}. However, ensuring that $v$ serves as a valid descent direction and provides sufficient decrease requires additional care in the algorithm design. Notably, the recent work \cite{liu2022minres} shows that the residual vector $r$ returned by $\MR$ can ensure such sufficient decrease naturally; this indicates that $\MR$ may be a more compatible tool for leveraging non-positive curvature in nonconvex optimization problems. While Liu and Roosta provide non-asymptotic second-order complexity guarantees for Newton-MR in \cite{liu2022newton}, to our knowledge, asymptotic results ensuring strict saddle point avoidance for $\MR$-based methods still seem to be unavailable. We note that probabilistic techniques to avoid saddle points can be employed irrespective of the method's order; see \cite{ge2015escaping,du2017gradient,jin2017escape,jin2018accelerated,jin2021nonconvex} for first-order and \cite{anandkumar2016efficient} for third-order variants. However, such ideas appear to be unsuitable in the local asymptotic analysis of second-order methods. 

\begin{table}[t]
\centering
\setlength{\tabcolsep}{5pt}
\NiceMatrixOptions{cell-space-limits=1.1pt}
{\footnotesize
\begin{NiceTabular}{|p{2.2cm}p{0.9cm}|p{1.8cm}p{0.9cm}p{0.9cm}p{1.2cm}p{1.4cm}|}%
 [ 
   code-before = 
    \rectanglecolor{lavender!30}{3-3}{15-3}
    \rectanglecolor{lavender!30}{3-5}{15-5}
    \rectanglecolor{lavender!30}{3-7}{15-7}
 ]
\toprule
\Block[c]{2-1}{Method} & \Block[c]{2-1}{Ref.} & \Block[c]{2-1}{Type of Analysis} & \Block[c]{2-1}{Saddle Avoid.} & \Block[c]{2-1}{EP-Free} & \Block{1-2}{Non-isolated Minima} & \\[-0.5ex] \cmidrule(lr){6-7} \\[-3.75ex]
 & & & & & \Block{1-1}{Conv.} & \Block{1-1}{Rate $\omega$}  \\ \Hline
\Block{1-1}{{TR$+$NC ${}^{\textcolor{bluep}{{(a)}}}$}} & \Block{1-1}{\cite{conn2000trust}} & \Block{1-1}{Asymptotic} & \Block{1-1}{\cmark} & \Block{1-1}{\xmark} & \Block{1-1}{\xmark}  & \Block{1-1}{--} \\
\Block{1-1}{TR-Newton-CG} & \Block{1-1}{\cite{curtis2021trust}} & \Block{1-1}{Non-Asymp.} & \Block{1-1}{\cmark} & \Block{1-1}{\cmark} & \Block{1-1}{\xmark} & \Block{1-1}{--} \\ 
\Block{1-1}{TR-Newton-CG} & \Block{1-1}{\cite{rebjock2024fast2}} & \Block{1-1}{Asymptotic} & \Block{1-1}{(\cmark)} & \Block{1-1}{\cmark} & \Block{1-1}{\cmark ${}^{\textcolor{bluep}{(b)}}$} & \Block{1-1}{$<1$ $(\text{\xmark})$} \\ \Hline[tikz=densely dotted]
\Block{1-1}{Newton$+$NC ${}^{\textcolor{bluep}{{(a)}}}$} & \Block{1-1}{\cite{gould2000exploiting,curtis2019exploiting}} & \Block{1-1}{(Non-)Asymp.} & \Block{1-1}{\cmark} & \Block{1-1}{\xmark} & \Block{1-1}{\xmark}  & \Block{1-1}{--} \\
\Block{1-1}{NCN} & \Block{1-1}{\cite{paternain2019newton}} & \Block{1-1}{Non-Asymp.} & \Block{1-1}{\cmark} & \Block{1-1}{\xmark} & \Block{1-1}{\xmark} & \Block{1-1}{--} \\
\Block{1-1}{Newton-CG} & \Block{1-1}{\cite{royer2020newton}} & \Block{1-1}{Non-Asymp.} & \Block{1-1}{\cmark} & \Block{1-1}{\cmark} & \Block{1-1}{\xmark} & \Block{1-1}{--} \\
\Block{1-1}{{Newton-MR}} & \Block{1-1}{\cite{liu2022newton}} & \Block{1-1}{Non-Asymp.} & \Block{1-1}{\cmark} & \Block{1-1}{\cmark} & \Block{1-1}{\xmark} & \Block{1-1}{--} \\ \Hline[tikz=densely dotted]
\Block{1-1}{{CR-Newton}} & \Block{1-1}{\cite{nesterov2006cubic}} & \Block{1-1}{(Non-)Asymp.} & \Block{1-1}{\cmark} & \Block{1-1}{-- ${}^{\textcolor{bluep}{(c)}}$} & \Block{1-1}{\cmark} & \Block{1-1}{$=\frac13$ $(= \frac13)$} \\
\Block{1-1}{{CR-Newton}} & \Block{1-1}{\cite{yue2019quadratic}} & \Block{1-1}{Asymptotic} & \Block{1-1}{\cmark} & \Block{1-1}{-- ${}^{\textcolor{bluep}{(c)}}$} & \Block{1-1}{\cmark} & \Block{1-1}{$=1$ $(=1)$} \\
\Block{1-1}{{ARC}} & \Block{1-1}{\cite{rebjock2024fast1}} & \Block{1-1}{Asymptotic} & \Block{1-1}{(\cmark)} & \Block{1-1}{-- ${}^{\textcolor{bluep}{(c)}}$} & \Block{1-1}{\cmark ${}^{\textcolor{bluep}{(b)}}$} & \Block{1-1}{$=1$ (\xmark)} \\ \Hline[tikz=densely dotted] \\[-2ex]
\Block[c]{1-2}{this work} & & \Block[c]{1-1}{Asymptotic} & \Block[c]{1-1}{\cmark} & \Block[c]{1-1}{\cmark} & \Block[c]{1-1}{\cmark} & \Block[c]{1-1}{$<1$ $(\leq \frac13)$} \\
\bottomrule 
\end{NiceTabular}
}
\caption{A summary of classical and recent works on second-order methods that can avoid saddles and/or converge to non-isolated minima that satisfy the (local) {\L}ojasiewicz inequality with exponent $\theta = \frac12$. Here, ``EP-Free'' denotes approaches that do not require (implicit) computations of eigenpairs. The parameter $\omega \in (0,1]$ (in the last column) characterizes the overall rate of superlinear convergence of $\{\|\nabla f(x_k)\|\}$, i.e., $\|\nabla f(x_{k+1})\| = \mathcal O(\|\nabla f(x_k)\|^{1+\omega})$, $k\to\infty$. The rate shown in parenthesis is the best rate under which each respective method is ensured to avoid (strict) saddle points. 
\endgraf
\setlength{\parindent}{1ex} \setlength{\baselineskip}{10pt}
${}^{\textcolor{bluep}{(a)}}$ {\footnotesize The algorithms in \cite{conn2000trust,gould2000exploiting,curtis2019exploiting} are (trust-region) Newton-type approaches that exploit negative curvature (NC) through eigenvalue-based tests and generation of proper directions.} \\
\hspace*{1pt} ${}^{\textcolor{bluep}{(b)}}$ {\footnotesize The results in \cite{rebjock2024fast1,rebjock2024fast2} require $x_k \to \bar x$ where $\bar x$ is a (non-isolated) local minimum.} \\
\hspace*{1pt} ${}^{\textcolor{bluep}{(c)}}$ {\footnotesize In \cite{nesterov2006cubic,yue2019quadratic}, the cubically regularized subproblems are assumed to be solved exactly; in \cite{rebjock2024fast1}, inexact solutions to such models are considered.} 
\vspace{-0.1ex}
\endgraf
}\label{table:sum}
\end{table}

Non-isolated minima frequently arise in overparameterized optimization or problems with continuous symmetries \cite{luo1993error,nesterov2006cubic,liu2022loss}. The cubic regularization method \cite{nesterov2006cubic} and its adaptive variant (ARC) \cite{cartis2011adaptive1,cartis2011adaptive2} can achieve local superlinear convergence under generalized conditions that do not require positive definiteness of the Hessian---see works like \cite{nesterov2006cubic,rebjock2024fast1} using the Polyak-{\L}ojasiewicz (PL) condition, \cite{zhou2018convergence} using the Kurdyka-{\L}ojasiewicz (KL) property and \cite{yue2019quadratic} using an error bound (EB) condition. Recently, in \cite{rebjock2024fast2}, Rebjock and Boumal establish, for the first time, superlinear convergence of the trust region method with CG as the inner solver under the (local) PL condition. Specifically, they provide a novel discussion of the performance of CG around non-isolated minima. As noted in \cite{rebjock2024fast2}, similar local convergence guarantees for $\MR$-based methods are not available. Moreover, the results in \cite{rebjock2024fast2} are local in nature and it is assumed that the iterates eventually enter a sufficiently small neighborhood of a minimum. In this context, it is interesting to ask whether comprehensive global-to-local convergence properties can be derived, ensuring that iterates can reach such local neighborhoods. \\[2mm]
\noindent\textit{Contributions.} Our idea is to apply MINRES to solve \eqref{sub: LS} and to leverage non-positive curvature information following the general paradigm of truncated methods. In summary, our contributions include:
\begin{itemize}
    \item We propose a MINRES-based, Newton-type algorithm that integrates (forward and backward) line-search globalization, built-in detection of non-positive curvature, and adaptive regularization. Under standard assumptions, we show that accumulation points of the generated sequence of iterates $\{x_k\}$ correspond to stationary points of the problem. Moreover, stronger convergence properties, including iterate convergence $x_k \to \bar x$, are established under the well-known Kurdyka-{\L}ojasiewicz (KL) inequality. Our analysis recovers classical results for truncated Newton methods, \cite{dembo1983truncated,dahito2019conjugate}, while relying on a more modern and weaker set of assumptions. To our knowledge, this is among the first works to provide in-depth asymptotic convergence guarantees for MINRES-based truncated algorithms.

    \item When the matrices $B_k$ in \eqref{sub: LS} are chosen as exact Hessians, we show that our algorithmic framework can avoid strict saddle points without additional costs or any eigenvalue-related computations. We further analyze the local convergence behavior of the proposed approach near non-isolated minima under the {\L}ojasiewicz property and establish superlinear convergence of $\{\|\nabla f(x_k)\|\}$ with a rate that is arbitrarily close to quadratic. These desirable guarantees are partly made possible by our algorithmic mechanisms to adaptively control and choose the tolerance, curvature test, and regularization parameters. 
    \item Finally, we illustrate the efficiency of the proposed MINRES-based algorithm and report its numerical performance on deep auto-encoder problems and the CUTEst test collection.
\end{itemize}

In \Cref{table:sum}, we provide an additional overview and comparison of related work and second-order methodologies. \\[2mm]
\noindent\textit{Notation.}
We use $\bf{0}$ to denote a vector or matrix of zeros. We use $ I_{t} $ to denote the identity matrix of dimension $ t \times t $ and $e_{j}^{(t)}$ is the $j$-th column of $I_{t}$. For $ t\geq 1 $, the Krylov subspace of degree $t$ generated using $ b $ and $A$ is defined as $\mathcal{K}_{t}(A, b) = \mathrm{span}\{ b,Ab,\ldots,A^{t-1}b\}$.  We use the subscripts $k$ to indicate the iterations of the main algorithm. The superscript $t$ is reserved to denote the iterations of a subproblem solver to obtain $d_k$, e.g., $d_k=p_{k}^{(t)}$ or $d_k=r_{k}^{(t-1)}$. We often use the abbreviations $d_k=p_t$, $d_k=r_{t-1}$, or $d_t = d_k^{(t)}$ (when the correspondence $p_t \equiv p_k^{(t)}$, $r_{t-1} \equiv r_k^{(t-1)}$ is clear). The residual vector at the $(t-1)$-th inner iteration is given by $r_{t-1}=b-Ap_{t-1}$. We further set $g_k \equiv \nabla f(x_k)\in \mathbb{R}^{n}$ and $B_k \approx \nabla^2 f(x_k) \in \mathbb{R}^{n \times n}$. For a symmetric matrix $A \in \mathbb S_n \subset \bR^{n \times n}$, we write $A \succeq (\succ)\;\! 0$ $A$ is positive semi-definite (positive definite). The Euclidean norm (for vectors) and the spectral norm (for matrices) are denoted by $\| \cdot\|$. \\[2mm]
\noindent\textit{Organization.} In \Cref{sec: MR}, we review MINRES and its main properties. In \Cref{sec: LS}, we present a MINRES-based linesearch algorithm and conduct a detailed convergence analysis. In \Cref{sec: Num}, we demonstrate the favorable numerical performance of the proposed algorithm. 

\section{Review of $\MR$} \label{sec: MR}
In this section, we briefly review $\MR$ and list several useful properties that will be relevant in our analysis. The interested reader is referred to \cite{liu2022minres} for further details. We consider a general symmetric (potentially indefinite or singular) matrix $A \in \mathbb S_n \subset \mathbb{R}^{n \times n}$ and a vector $b \in \mathbb{R}^{n}$. Moreover, for simplicity, we set $p_0=\mathbf{0}$, which implies $\mathcal{K}_{t}(A, r_0)=\mathcal{K}_{t}(A, b)$. 

\subsection{$\MR$: Review and Basic Components} \label{ssec: MR1} 
We first provide a definition of non-positive curvature ($\NPC$) directions. 

\begin{defn}\label{def:NPC}
  A vector $v\in \mathbb{R}^{n}\backslash \{\bf{0}\}$ satisfying $v^{\top}Av\le 0$ is called a non-positive curvature direction (for $A$). 
\end{defn}

The interplay between $A$ and $b$ is crucial to the convergence of $\MR$. This connection is entirely described by the grade of $b$ $\wrt$ $A$, see \cite{saad2003iterative}.

\begin{defn}\label{def:g}
The grade of $b$ $\wrt$ $A$ is the positive integer $g\in \mathbb{N}$ satisfying
\begin{equation*}
    \operatorname{dim}(\mathcal{K}_{t}(A, b)) = t \;\; \text{if} \;\; t \leq g \quad \text{and} \quad \operatorname{dim}(\mathcal{K}_{t}(A, b)) = g \;\; \text{if} \;\; t > g. 
\end{equation*}
\end{defn}
 
$\MR$ solves the following least-squares formulation of the problem \eqref{sub: LS}:
\begin{equation}\label{eq:MR1}
   p_t = {\argmin}_{p \in \mathcal{K}_{t}(A, b)}~\|b - A p\|.
\end{equation}
$\MR$ consists of three major steps, as shown in \Cref{alg: MINRES_LS}:
the Lanczos process, a QR decomposition, and the iterate update. \\[1mm] 
%
\noindent\textbf{Lanczos process.} Setting $v_1=b/\|b\|$, after $t$ iterations, the Lanczos vectors form an orthogonal matrix $ V_{t+1} = \begin{bmatrix} v_1  & v_{2} & \dots &  v_{t+1} \end{bmatrix} \in \mathbb{R}^{n \times (t+1)} $, whose columns span $\mathcal{K}_{t+1}(A, b)$ and satisfy  
\begin{equation}\label{eq: AV1}
    A V_{t} = V_{t+1} \tilde{T}_{t}=V_{t+1} \begin{bmatrix}
    T_{t} \\
    \beta_{t+1}(e_{t}^{(t)})^{\top}
\end{bmatrix}=V_{t+1} \begin{bmatrix} V_{t}^{\top} A V_{t} \\ \beta_{t+1}(e_{t}^{(t)})^{\top} \end{bmatrix},
\end{equation}
where $ \tilde{T}_{t} \in \bR^{(t+1) \times t}$ is an upper Hessenberg matrix and $\beta_{t+1} \in  \bR$ is a scalar. Let $p_t = V_t q_t$ be given for some $q_t \in\mathbb{R}^{t}$. We can then rewrite \eqref{eq:MR1} as
\begin{equation}\label{eq:MR2}
    {\min}_{q_{t} \in \mathbb{R}^{t}}~\|\beta_{1} e_1^{(t+1)} - \tilde{T}_{t} q_{t}\|, \quad \beta_{1} = \|b\|.
\end{equation}
\noindent\textbf{QR decomposition.} The least-squares problem \eqref{eq:MR2} is solved via QR decomposition of $\tilde{T}_{t}$. It holds that
\begin{equation*}
    Q_{t} \tilde{T}_{t} = \tilde{R}_{t} \equiv 
    \begin{bmatrix}
        R_{t}\\
        \zero^{\top}
    \end{bmatrix},  \quad   \min_{q_{t} \in \mathbb{R}^{t}} \|\beta_{1} e_1^{(t+1)} - \tilde{T}_{t} q_{t}\| = \min_{q_{t} \in \bR^{t}} \vnorm{\begin{bmatrix}
        t_{t}\\
        \phi_{t}
    \end{bmatrix} - 
    \begin{bmatrix}
        R_{t} \\
        \zero^{\top}
    \end{bmatrix}q_{t}},
\end{equation*}
where $ Q_{t} \in \bR^{(t+1) \times (t+1)} $ is orthogonal, 
$R_{t} \in \bR^{t \times t}$ is upper-triangular, and we have $Q_{t}\beta_1 e_1^{(t+1)} = [t_t; \phi_{t}]^{\top}$, where $t_t\in \bR^{t}$, $\phi_{t} \in \bR$. Hence, $q_t$ satisfies $R_t q_t =t_t$, which yields $\phi_t=\|r_t\|$. We also trivially have $\phi_0=\beta_1=\|b\|$. \\[1mm]
\noindent\textbf{Iterate updates.} Suppose $ t < g $ and let $ D_{t} $ be given by the lower-triangular system $ R^{\top}_{t} D^{\top}_{t} = V^{\top}_{t} $. Now, setting $ V_{t} = [ V_{t-1} \mid v_{t}] $ and using the fact that $ R_{t} $ is upper-triangular, we obtain the recursion $ D_{t} = [ D_{t-1} \mid d_{t}] $ for some vector $ d_{t}  \in \bR^n$. As a result, using $ R_{t} q_{t} = t_{t} $, we can update the iterate $p_{t-1}$ via
\begin{align*}
	p_{t} = V_{t} q_{t} = D_{t} R_{t} q_{t} = D_{t} t_{t} = \begin{bmatrix}
		D_{t-1} & \mid & d_{t}
	\end{bmatrix} \begin{bmatrix}
		t_{t-1} \\
		\tau_{t}
	\end{bmatrix} = p_{t-1} + \tau_{t} d_{t}.
\end{align*}
Here, $\tau_t$ is a scalar that can be computed from previous iterations. By $ R^{\top}_{t} D^{\top}_{t} = V^{\top}_{t}$, the direction $d_t$ can be updated recursively using $\{v_i\}_{i=1}^{t-1}$ and $\{d_i\}_{i=1}^{t-1}$. 

\subsection{Basic Properties of $\MR$}\label{ssec: MR2}
Next, we summarize three core properties of $\MR$: the $\NPC$ detection, a certificate for positive definiteness of $A$, and other useful results. We will omit proofs and refer interested readers to \cite{liu2022minres,saad2011numerical}. Theorem~\ref{thm: NPC} essentially implies that $\NPC$ directions can be detected without additional costs in $\MR$. 

\begin{algorithm}[t]
	\caption{\textbf{$\MR$ with Built-in $\NPC$ Detection for \eqref{sub: LS} }} 
    \label{alg: MINRES_LS}
	\begin{algorithmic}[1]
		\State \textbf{Input:} the matrix $ \AA $, the vector $\bb$, the tolerance parameter $\theta$,
		\State $ \beta_1 = \vnorm{\bb} $, $ \rr_{0} = \bb $, $ \vv_1 = \bb /\beta_1 $, $ \vv_0 = \xx_0 = \dd_0 = \dd_{-1} = \zero $, $ c_0 = -1 $, $ s_0 = \delta^{(1)}_1= \epsilon_1= 0 $, $ \phi_0  = \beta_1 $, $  t= 1 $;
		\While {true}
		\State $ \pp_{t} = \AA \vv_{t}$, $\alpha_{t} = \vv_t^{\top}\pp_{t} $, $ \pp_{t} = \pp_{t} - \beta_{t} \vv_{t-1} $, $ \pp_{t} = \pp_{t} - \alpha_{t} \vv_{t} $, $ \beta_{t+1} = \vnorm{\pp_{t}} $,
		\State $\left[\begin{array}{ll}\delta_t^{(2)} & \epsilon_{t+1} \\ \gamma_t^{(1)} & \delta_{t+1}^{(1)}\end{array}\right]=\left[\begin{array}{cc}c_{t-1} & s_{t-1} \\ s_{t-1} & -c_{t-1}\end{array}\right]\left[\begin{array}{cc}\delta_t^{(1)} & 0 \\ \alpha_t & \beta_{t+1}\end{array}\right]$;     
        \If{$ c_{t-1} \gamma^{(1)}_{t} \geq 0$} 
            \State  flag\,$=$\,$\NPC$, \Return $d =  \beta_1\frac{\rr_{t-1}}{\|\rr_{t-1}\|}$ \text{as an $\NPC$ direction}; 
            \EndIf
		\State $ \gamma_{t}^{(2)} = \sqrt{(\gamma_{t}^{(1)})^2 + \beta_{t+1}^2} $;
		\State $c_{t} = \gamma_{t}^{(1)} / \gamma_{t}^{(2)}$\footnotemark{}, $ s_{t} = \beta_{t+1} / \gamma_{t}^{(2)} $, $ \tau_{t} = c_{t} \phi_{t-1} $, $ \phi_{t} = s_{t} \phi_{t-1} $;
		\State $ \dd_{t} = (\vv_{t} - \delta^{(2)}_{t} \dd_{t-1} - \epsilon_{t} \dd_{t-2} ) / \gamma^{(2)}_{t} $, $ \xx_{t} = \xx_{t-1} + \tau_{t} \dd_{t} $;
        \If{$ \phi_{t} \le \theta \beta_1 $}
        \State  flag\,$=$\,SOL, \Return $ d= \xx_{t} $ as a solution to \eqref{sub: LS};
        \EndIf   
		\State $ \vv_{t+1} = \pp_{t} / \beta_{t+1} $\footnotemark{}, $ \rr_{t} = s_{t}^2 \rr_{t-1} - \phi_{t} c_{t} \vv_{t+1} $;
		\State $ t \leftarrow t+1 $;
		\EndWhile
		\State \textbf{Output:}  $\NPC$ or SOL direction $d$, $r_t$, flag;
    	\end{algorithmic}
\end{algorithm} 

\footnotetext[1]{Here, we have $\gamma_{t}^{(2)} \neq 0$. Note that $\gamma_{t}^{(2)} = 0$ implies $\gamma_{t}^{(1)} = 0$ by Line 9. Hence, $\NPC$ will be detected and the algorithm terminates at Line 7.}
\footnotetext[2]{Here, we have $\beta_{t+1} \neq 0$. Note that if $\beta_{t+1} = 0$, then $s_t=0$, and $\phi_t=0$ at Line 10. Hence, the algorithm terminates earlier at Line 13.}

\begin{thm}[$\NPC$ detection] 
    \label{thm: NPC}
    Let $ g $ be the grade of $b $ w.r.t. $ A $. Assume that the residual vectors $\{r_t\}$ are generated by $\MR$. It holds that 
    \begin{subequations}		
		\begin{align}
			r_{t-1}^{\top}A r_{t-1} &= -\phi_{t-1}^2 c_{t-1} \gamma^{(1)}_{t},   \quad\;\!\, 1 \leq t \leq g, \label{eq:rAr} \\		
			r_t^{\top} b &= \vnorm{r_{t}}^2,  \quad\quad\quad\quad\quad 1 \leq t \leq g. \label{eq:bTrk}
		\end{align}
	\end{subequations}
\end{thm}


Here, $\phi_{t-1}$, $c_{t-1}$ and $\gamma^{(1)}_{t}$ are parameters that are computed in $\MR$, cf$.$ \Cref{alg: MINRES_LS}. From \eqref{eq:rAr}, if $\phi_{t-1}^2 c_{t-1} \gamma^{(1)}_{t} \ge 0$, then we have $r_{t-1}^{\top}A r_{t-1} \le 0$ and $r_{t-1}$ can be declared as an $\NPC$ direction. 
In classical truncated methods, such as, e.g., Newton-CG, \cite{dembo1983truncated,royer2020newton}, and Newton-CR, \cite{dahito2019conjugate}, such directions are also not hard to obtain but they are not necessarily sufficient descent directions. By contrast, \eqref{eq:bTrk} states that the residual vector in $\MR$ can be an $\NPC$ direction yielding sufficient decrease for $b=-g_k$.

We note that the matrix $T_{t}  =  V_{t}^{\top} A V_{t} \in \bR^{t \times t}$ in \eqref{eq: AV1} has a strong connection to $A$ and plays an important role in the analysis of $\MR$. 

\begin{thm}[Certificate of positive definiteness] \label{thm: Cer}
    Given $A\in\mathbb{S}_{n}$ and $b\in\mathbb{R}^{n}$ and recalling $g$ and $T_t$ (see Definition \ref{def:g} and \eqref{eq: AV1}), it holds that:
    \begin{enumerate}[label=\textup{(\roman*)},topsep=1pt,itemsep=0ex,partopsep=0ex, leftmargin = 25pt]
        \item We have $T_t \succ0$ as long as $\NPC$ is not detected in $\MR$. 
        \item The $i$-th eigenvalue of $T_t$ is upper bounded by the $i$-th eigenvalue of $A$.
        \item $\MR$ returns a solution to \eqref{eq:MR1} after exactly $ g $ iterations (in exact arithmetic).
    \end{enumerate} 
\end{thm}
 
The matrix $T_t$ is known to converge to $A$ if exact arithmetic is used; see \cite[Section 6.6.2]{saad2011numerical}. Part (iii) of Theorem~\ref{thm: Cer} essentially follows from the definition of $g$ demonstrating the efficiency of $\MR$ when applied to low-grade systems. 

As long as $\NPC$ has not been detected, the iterates $\{p_t\}$ enjoy the following property, which will be frequently used in the convergence analysis of our framework; cf$.$ \cite[Theorems 3.8 and 3.11]{liu2022minres}.

\begin{thm}[Properties of $\MR$]\label[theorem]{lemmaMR}
    Let $ g $ be the grade of $b $ w.r.t. $ A $. As long as $\NPC$ has not been detected, the iterates $\{p_t\}$, $1 \leq t < g $, satisfy:
    \begin{enumerate}[label=\textup{(\roman*)},topsep=1pt,itemsep=0ex,partopsep=0ex, leftmargin = 25pt]
        \item $p_t^{\top} b >  p_t^{\top} A  p_t.$
        \item $p_1 = \frac{b^{\top}Ab}{\|{A b}\|^2} b$ and $ p_t^{\top}b\ge p_1^{\top}b = \frac{b^{\top}Ab}{\|{A b}\|^2} \|b\|^2$. 
    \end{enumerate} 
 \end{thm}

\section{A $\MR$-based Linesearch Algorithm}\label{sec: LS}
Our algorithmic framework is inspired by linesearch-based, truncated Newton methods. In contrast to previous works, \cite{dahito2019conjugate,dembo1983truncated}, which utilize CR or CG to generate truncated Newton directions, we intend to leverage the favorable properties of $\MR$ for this purpose. Specifically, at the $k$-th iteration, we apply $\MR$ to compute the direction $d_k$ by solving
\begin{equation*}\label{eq:sub: LS_k}
    {\min}_{d \in \mathbb{R}^n}~\| B_k d + g_k \|.
\end{equation*}
The full algorithmic framework and global convergence results are discussed in \Cref{ssec: LS1,ssec: LS2,ssec: KL}. In \Cref{ssec: LS3,ssec: LS4}, we verify that the proposed $\MR$-based approach effectively avoids strict saddle points 
and we present local convergence results to potentially non-isolated minima.

\begin{algorithm}[t]
    \caption{\textbf{$\MR$-based Linesearch Algorithm}}
    \label{alg:main_LS}
    \begin{algorithmic}[1] 
        \State \textbf{Input:} initial point $x_0\in \mathbb{R}^n$, parameters $\rho, \sigma \in (0,1)$, $\alpha \in (0,1]$,  $\cA, s>0$, tolerances $\{\theta_k\} $, parameter sequences $\{a_k\}$, $\{\zeta_k\}$. 
        \For{$k=0,1,...$}
            \State Compute $g_k=\nabla f(x_k)$, select $B_k \in \mathbb{S}_n$, and set $\bar B_k = B_k + \zeta_k I$;
            \If{$\|g_k\| = 0$}
            \State \Return $x^*=x_k$;
            \EndIf
             \State $(d_k,r_{k},\text{flag}_k)= \MR(\bar B_k,-g_k,\theta_k)$ by \Cref{alg: MINRES_LS};
            \If{$\text{flag}_k$\,$=$\,SOL and $d_k^{\top}\bar B_k d_k < \min\{\cA, a_k \|g_k\|^{\alpha} \}\|d_k\|^{2}$}
             \State $d_k=-g_k$, $\text{flag}_k$\,$=$\,GD;
             \EndIf
             \State Compute a step size $\lambda_k$ using \Cref{alg: LS} that satisfies \eqref{armijo} or \eqref{armijo2};
             \State Set $x_{k+1}=x_k + \lambda_k d_k$, $k \leftarrow k+1$;     
        \EndFor
    \end{algorithmic}
\end{algorithm}


\begin{algorithm}[t]
    \caption{\textbf{Linesearch Scheme}}
    \label{alg: LS}
    \begin{algorithmic}[1] 
        \State \textbf{Input:} current iterate $x_k\in \mathbb{R}^n$, a direction $d_k$, parameters $\rho, \sigma \in (0,1)$, $s>0$;
        \State Set $\lambda=s$;
        \If{flag\,$=$\,SOL or GD}    
        \State \textbf{while} {\eqref{armijo} does not hold for $\lambda$} \textbf{do}
             $\lambda=\rho\lambda$; \textbf{end while} 
        \ElsIf{flag\,$=$\,$\NPC$}
            \If{\eqref{armijo2} is not satisfied for $\lambda=s$} 
                \State \textbf{while} {\eqref{armijo2} does not hold for $\lambda$} \textbf{do}
                $\lambda=\rho\lambda$; \textbf{end while} 
            \Else
                \State \textbf{while} {\eqref{armijo2} holds  for $\lambda$}  \textbf{do} $\lambda=\lambda/\rho$; \textbf{end while}
                \State $\lambda = \rho\lambda$;
            \EndIf
        \EndIf
        \State \Return $\lambda_k=\lambda$;
    \end{algorithmic}
\end{algorithm}


\subsection{Algorithm Framework}\label{ssec: LS1}
The full framework is shown in \Cref{alg:main_LS}, which calls \Cref{alg: MINRES_LS} and \Cref{alg: LS}. There are two main features. Firstly, similar to \cite{dahito2019conjugate,dembo1983truncated}, we use three types of directions: the solution direction $p_t$, the scaled $\NPC$ direction $\frac{\|g_k\|}{\|r_{t-1}\|}r_{t-1}$, and the negative gradient direction $-g_k$. Given the parameters $\theta_k$, $\zeta_k \geq 0$, we first use \Cref{alg: MINRES_LS}, 
\begin{equation*}\label{eq:regNewton}
    (d_k, r_k,\text{flag}_k)=\MR(\bar{B}_k, -g_k, \theta_k), \quad \bar{B}_k := B_k + \zeta_k I,  
\end{equation*}
to compute an inexact solution $d_k$ of a regularized Newton system.
By \eqref{eq:rAr}, if $c_{t-1} \gamma^{(1)}_{t} \geq 0$, then $\NPC$ is detected and $\frac{\|g_k\|}{\|r_{t-1}\|}r_{t-1}$ is an $\NPC$ direction with $r_{t-1}^{\top}B_kr_{t-1}\le -\zeta_k \|r_{t-1}\|^2$, where $-\zeta_k$ can be interpreted as a pseudo-negative eigenvalue of $B_k$. The regularization parameter $\zeta_k$ plays an important role in the avoidance of strict saddle points and the local convergence. 
If $d_k = p_t$, we perform the following additional curvature test
\begin{equation}\label{eq:cur}
    p_t^{\top}\bar{B}_k p_t \ge \min\{\cA, a_k \|g_k\|^\alpha\}\|p_t\|^2, \quad \cA, \alpha \in (0,1], \; a_k \geq \bar a > 0, 
 \end{equation}
in \Cref{alg:main_LS}. Proper choice of $\{a_k\}$ allows us to apply Kurdyka-{\L}ojasiewicz-based analysis techniques and to obtain unified global and local convergence results. 
We set $d_k=-g_k$ if \eqref{eq:cur} fails. The test \eqref{eq:cur} aims to avoid small positive curvature. 
Due to $r_{t}=-g_k - \bar B_k p_t$, it follows $p_t^{\top}\bar B_k p_t = -p_t^{\top} (g_k + r_t)$ for all $t$. Thus, \eqref{eq:cur} can be checked efficiently if we return $r_t$. 

Secondly, we use two different linesearch strategies depending on the final choice of $d_k$. If $d_k=p_t$ or $d_k = -g_k$, we choose the step size $\lambda_k=s \rho^j$ by \textit{backtracking} to satisfy the Armijo condition
\begin{equation}\label{armijo}
    f(x_{k}+ \lambda_k d_k)-f(x_{k}) \le \sigma \lambda_k d_k^{\top}g_k,
\end{equation}
where $s>0$, $\rho, \sigma \in (0,1)$ and $j$ is the smallest nonnegative integer such that \eqref{armijo} holds. 
In the $\NPC$ case, $d_k = \frac{\|g_k\|}{\|r_{t-1}\|}r_{t-1}$, we check the condition:
\begin{equation}\label{armijo2}
   f(x_k +\lambda_k d_k) -f(x_k) \le  \sigma \lambda_k g_k^{\top} d_k + \frac{\sigma}{2} \lambda_k^2 d_k^{\top} B_k d_k.
\end{equation}
When \eqref{armijo2} is satisfied for $\lambda_k=s$, we increase $\lambda_k=s \rho^{-j}$ iteratively by the factor $1/\rho >1$, $j = 1,2,\dots$, such that \eqref{armijo2} still holds (i.e., here, $j$ is the largest integer that satisfies the linesearch condition). This is a \emph{forward linesearch}. If \eqref{armijo2} fails for $\lambda_k=s$, we use backtracking instead. Our main motivation is to force the step to move away from the current position (as the current iterate may lie in the region of a strict saddle point). Due to $d_k^{\top} B_k d_k \le -\zeta_k \|d_k\|^2 \le 0$, \eqref{armijo2} can be seen as a stricter version of \eqref{armijo}. The forward linesearch and the modified condition \eqref{armijo2} will play an integral role in the saddle point avoidance, cf. \Cref{ssec: LS3}.  Similar strategies have also been used in~\cite{gould2000exploiting,liu2022newton}.

\subsection{Global Convergence Analysis} \label{ssec: LS2}
We first state some assumptions and preparatory lemmas before introducing the main theorem. Our analysis is based on the key assumptions:
\begin{mdframed}[style=assumptionbox]
\begin{enumerate}[label=\textup{\textrm{(A.\arabic*)}},topsep=2pt,itemsep=0ex,partopsep=0ex,leftmargin=6ex]
    \item \label{A.1} The function $f:\bR^n\to\bR$ is twice continuously differentiable. 
\end{enumerate}
\begin{enumerate}[label=\textup{\textrm{(B.\arabic*)}},topsep=2pt,itemsep=0.1ex,partopsep=0ex,leftmargin=6ex]
    \item \label{B.1} There is $M>0$ such that $\|B_k\| \le M$ for all $k \in \mathbb{N}$.  
    \item \label{B.2} The tolerance parameters $\{\theta_k\}$ satisfy $\theta_k = \min\{\cB,b_k\|g_k\|^\beta\}$ for some $\cB, \beta >0$, $b_k \geq \bar b > 0$, and all $k$.   
    \item \label{B.3} We set $\zeta_k = \min\{\cZ,z_k\|g_k\|^\zeta\}$ with $\cZ > 0$, $\zeta \in (0,1]$, $z_k \geq \bar z > 0$. 
\end{enumerate}
\end{mdframed}

We will also work with the following alternative version of \ref{B.1}:
\begin{mdframed}[style=assumptionbox]
\begin{enumerate}[label=\textup{\textrm{(B.\arabic*)${}^\prime$}},topsep=2pt,itemsep=0ex,partopsep=0ex,leftmargin=7ex]
    \item \label{B.4} For all $k \in \mathbb{N}$, it holds that $B_k = \nabla^2 f(x^k)$.
\end{enumerate}
\end{mdframed}


Before starting our analysis, some remarks on the well-definedness of Algorithms~\ref{alg:main_LS} and \ref{alg: LS} are in order. Due to the forward linesearch mechanism in Algorithm~\ref{alg: LS}, the method might stop after finitely many iterations, in a loop, with a direction along which $f$ approaches $-\infty$. In the following, we show that this situation can be avoided when $f$ is assumed to be bounded from below. 
\begin{lem} \label[lemma]{lem:well-defined}
Let $f : \bR^n \to \bR$ be continuously differentiable and let $x, d \in\mathbb{R}^n$, $B \in \mathbb{S}_n$, and $\sigma \in (0,1)$ be given with $\|d\| \neq 0$, $\nabla f(x)^\top d <0$ and $d^{\top}Bd \le 0$. 
\begin{enumerate}[label=\textup{(\roman*)},topsep=1ex,itemsep=0ex,partopsep=0ex, leftmargin = 25pt]
\item There exists $\bar \lambda_\ell = \bar \lambda_\ell(x,d,B) > 0$ such that 
\begin{equation} \label{eq:lemma-linesearch} f(x+\lambda d)-f(x) \leq \sigma \lambda \nabla f(x)^\top d + \frac{\sigma}{2} \lambda^2 d^\top B d \quad \forall~\lambda \in [0,\bar\lambda_\ell]. \end{equation}  
\item In addition, suppose that $f$ is bounded from below. Then there is  $\bar \lambda_u =\bar \lambda_u(x,d,B) > 0$ such that \eqref{eq:lemma-linesearch} does not hold for all $\lambda > \bar\lambda_u$.
\end{enumerate} 
\end{lem}
\begin{proof} Part (i) readily follows from
\[ \lim_{\lambda \to 0} \frac{f(x+\lambda d)-f(x)- \sigma\lambda \nabla f(x)^{\top}d - \frac{\sigma}{2} \lambda^2  d^{\top} B d}{\lambda} = (1-\sigma) \nabla f(x)^{\top}d < 0; \]
see also \cite[Lemma 2.1]{gould2000exploiting}. Part (ii) is clear.
\end{proof}
Under the assumptions in \Cref{lem:well-defined}, Algorithms~\ref{alg:main_LS} and \ref{alg: LS} are well-defined, i.e., the linesearch procedures in Algorithm~\ref{alg: LS} will terminate after finitely many steps. In the following, whenever Algorithm~\ref{alg:main_LS} is said to generate a sequence of iterates $\{x_k\}$, we implicitly assume that the forward linesearch in Algorithm~\ref{alg: LS} stops within finitely many iterations, i.e., we will not explicitly assume lower boundedness of $f$. While it is possible to introduce additional adaptive upper bounds for $\lambda_k$ (in Step 9 of Algorithm~\ref{alg: LS}), we will not pursue such an approach here for the ease of exposition.

Under \ref{A.1}--\ref{B.1}, the Newton-CG (and -CR) method (using Powell-Wolfe linesearch) generates directions $d_k$ satisfying $-d_k^{\top}g_k \ge C_k^1 \|g_k\|^2$ and $\|d_k\|\le C_k^2 \|g_k\|$ for some $C_k^1, C_k^2 > 0$, see \cite{dahito2019conjugate,dembo1983truncated,steihaug1981quasi}. Next, we establish similar relationships between $-d_k^{\top}g_k$, $\|d_k\|$, and $\|g_k\|$ when $d_k=p_t$ is obtained by $\MR$ and $\lambda_k$ is chosen by backtracking.

\begin{lem}[Properties of SOL-directions]
   \label[lemma]{lemma: LS_SOL}
   Let $d_k=p_t$ be generated by \Cref{alg:main_LS} (i.e., in Step 11 of \Cref{alg:main_LS}, we have $\mathrm{flag}_k = \mathrm{SOL}$) and set 
   \[ s_k := x_{k+1}-x_k = \lambda_k d_k, \quad C_k := (\|\bar B_k\| + \|\bar B_k\|^2)^{-1}. \]
   %
   It then holds that:
   \begin{enumerate}[label=\textup{(\roman*)},topsep=1ex,itemsep=0ex,partopsep=0ex, leftmargin = 25pt]
       \item $-d_k^{\top}g_k > C_k {\min\{\cA, a_k\|g_k\|^{\alpha}\}} \|g_k\|^2$. \vspace{1ex}
       \item $\|d_k\|=\|p_t\| \le \max\{{\cA}^{-1} \|g_k\|, a_k^{-1}\|g_k\|^{1-\alpha}\}$. 
       \item If $\alpha<1$, then $f(x_k) -f(x_{k+1}) \ge C_k\sigma \|g_k\| \min\{\cA^2 \|s_k\|,  (\frac{a_k}{s^\alpha})^\frac{2}{1-\alpha}\|s_k\|^{\frac{1+\alpha}{1-\alpha}} \}$. 
   \end{enumerate} 
\end{lem}

\begin{proof}
    First, by Theorem~\ref{lemmaMR}~(ii), we have $p_1 = -\frac{g_k^{\top}\bar{B}_k g_k}{\|\bar{B}_k g_k\|^2} g_k$ and 
    \begin{equation}\label{eq:pkgk}
      -p_t^{\top} g_k \ge -p_1^{\top} g_k = \frac{g_k^{\top}\bar{B}_k g_k}{\|\bar{B}_k g_k\|^2} \|g_k\|^2, \quad t\ge 1.  
    \end{equation}
    When $d_k=p_1$, i.e., $\phi_1 \le \theta_k \beta_1$ in \Cref{alg: MINRES_LS}, then the right hand side of the inequality \eqref{eq:pkgk} can be bounded by the curvature test \eqref{eq:cur} as follows:
    \begin{equation*} 
       \frac{g_k^{\top}\bar{B}_k g_k}{\|\bar{B}_k g_k\|^2} \|g_k\|^2 \ge \frac{\min\{\cA, a_k\|g_k\|^{\alpha}\}}{\|\bar{B}_k\|^2} \|g_k\|^2. 
    \end{equation*}
    Let us now consider the case $d_k=p_t$ for $t\ge 2$, i.e., $\phi_1 > \theta_k\beta_1$ in \Cref{alg: MINRES_LS}. Using $V_{1} = \begin{bmatrix} v_1  \end{bmatrix} \in \mathbb{R}^{n \times 1} $, $p_1 = V_1 q_1 = -\frac{g_k^{\top}\bar{B}_k g_k}{\|\bar{B}_k g_k\|^2} g_k$, we obtain
    \begin{equation} \label{eq:nice}  \frac{g_k^{\top}\bar{B}_k g_k}{\|\bar{B}_k g_k\|^2} = \frac{p_1^{\top}\bar{B}_k p_1}{\|\bar{B}_k p_1\|^2} = \frac{q_1^{\top} V_1^{\top} \bar{B}_k V_1 q_1}{\|\bar{B}_k V_1q_1\|^2}. \end{equation}
    By construction and definition of the Krylov subspace, we have $\bar{B}_k V_t q_t \in \mathcal{K}_{t+1}(\bar{B}_k,-g_k) $ and $V_{t+1} V_{t+1}^{\top}$ is an orthogonal projector onto $\mathcal{K}_{t+1}(\bar{B}_k,-g_k)$. Thus, it holds that $V_{t+1} V_{t+1}^{\top} \bar{B}_k V_t q_t = \bar{B}_k V_t q_t$. Moreover, due to $\mathcal{K}_{t}(\bar{B}_k,-g_k) \subset \mathcal{K}_{t+1}(\bar{B}_k,-g_k)$, we have $ V_t q_t =  V_{t+1} \tilde{q}_{t}$  where $\tilde{q}_{t} = [q_t; 0] \in \mathbb{R}^{t+1}$. In addition, it holds that $T_{t}= V_{t}^{\top} \bar{B}_k V_{t}$ by the Lanczos process. We now use these properties in \eqref{eq:nice} when $t=1$; this yields
    $$ \frac{q_1^{\top} V_1^{\top} \bar{B}_k V_1 q_1}{\|\bar{B}_k V_1q_1\|^2}
    =\frac{\tilde{q}_1^{\top}V_2^{\top} \bar{B}_k V_2 \tilde{q}_1}{\|V_2^{\top} \bar{B}_k V_2 \tilde{q}_1\|^2}
    =\frac{\tilde{q}_1^{\top}T_2\tilde{q}_1}{\|T_2\tilde{q}_1\|^2}.$$
    Due to $d_k=p_t$, $t \ge 2$, it follows $r_0^{\top}\bar{B}_k r_0> 0 $ and $r_1^{\top}\bar{B}_k r_1 >0$, i.e., $\NPC$ is not detected during the first two iterations. By Theorem~\ref{thm: Cer}~(i), we then have $T_2 \succ 0$. Let $T_2= Q \Sigma Q^{\top}$ be an eigendecomposition of $T_2$, where $\Sigma =\mathrm{diag}(\lambda_1, \lambda_2)$ with $\lambda_1, \lambda_2 >0$ and $Q \in \mathbb{R}^{2\times 2}$ is an orthogonal matrix. Hence, setting $ u = [u_1;u_2] = Q^{\top} \tilde{q}_1 $, we obtain
    \begin{equation*}\label{eq:1/Bk}
    \frac{\tilde{q}_1^{\top}T_2\tilde{q}_1}{\|T_2\tilde{q}_1\|^2}
    =\frac{\sum_{i=1}^{2} u_i^2 \lambda_i }{\sum_{i=1}^{2} u_i^2 \lambda_i^2} 
    \ge \frac{\sum_{i=1}^{2} u_i^2 \lambda_i }{\max\{\lambda_1,\lambda_2\} \sum_{i=1}^{2} u_i^2 \lambda_i}
    =\frac{1}{\|T_2\|} \ge \frac{1}{\|\bar{B}_k\|}. 
    \end{equation*}
    The last inequality is due to Theorem~\ref{thm: Cer}~(ii). Using $\cA \in (0,1]$, this finishes the proof of part (i). We continue with part (ii). By \eqref{eq:bTrk}, we have for $t \ge 1$
    \begin{align} \nonumber
      \|\bar{B}_k p_t\|^2 &= \|-g_k-r_t\|^2= \| g_{k}\|^{2}+2g_{k}^{\top}r_t +\|r_t\|^{2} \\
      &= \| g_{k}\|^{2}-\|r_t\|^{2} \le \|g_k\|^2. \nonumber
    \end{align}
    This implies $\min\{\cA, a_k \|g_k\|^{\alpha} \} \|p_t\|^{2} \le p_t^{\top}\bar{B}_k p_t \le \|p_t\| \| \bar{B}_k p_t\| \le \|p_t\|\|g_k\|$, i.e., it holds that $\|d_k\| = \|p_t\| \le \max\{{\cA}^{-1} \|g_k\|, a_k^{-1} \|g_k\|^{1-\alpha}\}$, which proves part (ii).
    %
    If $\cA \leq a_k \|g_k\|^\alpha$, it follows $\|d_k\| \leq {\cA^{-1}}\|g_k\|$ and by \eqref{armijo} and part~(i), we have 
    $(f(x_k)-f(x_{k+1}))/ (\sigma \|g_k\|) \ge C_k\lambda_k \cA \|g_k\| \geq C_k \cA^2 \|s_k\|$.
    Otherwise, if $\cA \geq a_k \|g_k\|^\alpha$, we can infer $\|d_k\| \leq {a_k^{-1}} \|g_k\|^{1-\alpha}$ and due to $\lambda_k \leq s$, $\alpha \in (0,1)$, it holds that
    \begin{align*}
        (f(x_k)-f(x_{k+1}))/ (\sigma \|g_k\|) & \ge C_k\lambda_k a_k \|g_k\|^{1+\alpha} \geq C_k\lambda_k a_k^{\frac{2}{1-\alpha}}\|d_k\|^{\frac{1+\alpha}{1-\alpha}} \\ & \hspace{-4ex} = C_k\lambda_k^{-\frac{2\alpha}{1-\alpha}} a_k^{\frac{2}{1-\alpha}}\|s_k\|^{\frac{1+\alpha}{1-\alpha}} \geq C_k (a_k/s^\alpha)^{\frac{2}{1-\alpha}}\|s_k\|^{\frac{1+\alpha}{1-\alpha}}.
    \end{align*}
    Combining both estimates, this completes the proof of part~(iii).
    %
    %
    \end{proof}

    \begin{rem}\label{remark}
        The proof of Lemma~\ref{lemma: LS_SOL}~(i) discusses the cases $t=1$ and $t\ge2$ separately, which may seem unnatural. However, if $t=1$ and $\bar{B}_k$ is indefinite, we may have $g_k^{\top}\bar{B}_k g_k / {\|\bar{B}_k g_k\|^2}  < {1}/{\|\bar{B}_k\|}$, even though $g_k^{\top} \bar{B}_k g_k >0$, i.e., $\NPC$ is not detected at step $t=1$. For instance, let us consider $\bar B_k = \mathrm{diag}(1,-1)$, $-g_k = (a,b)^\top$, and $|a| > |b|$.
        Then, it follows $g_k^{\top}\bar{B}_k g_k= a^2 -b^2 >0$ and we can easily compute $T_1=v_1^{\top} \bar{B}_k v_1 = {g_k^{\top}\bar{B}_k g_k}/{\|g_k\|^2} >0$ and ${g_k^{\top}\bar{B}_k g_k}/{\| \bar{B}_k g_k\|^2} ={(a^2-b^2)}/{(a^2+b^2)} <  1 = {1}/{\|\bar{B}_k\|}$.
    \end{rem}

If $d_k$ is set to $\frac{\|g_k\|}{\|r_{t-1}\|}r_{t-1}$ or $-g_k$, we can derive analogous bounds. 

\begin{lem}[Properties of $\NPC$-directions]
    \label[lemma]{lemma: LS_NC}
    Let $d_k=\frac{\|g_k\|}{\|r_{t-1}\|}r_{t-1}$ be generated by \Cref{alg:main_LS} (i.e., in Step 11, we have $\mathrm{flag}_k = \NPC$). It holds that:
    \begin{enumerate}[label=\textup{(\roman*)},topsep=1ex,itemsep=0ex,partopsep=0ex, leftmargin = 25pt]
        \item $-d_k^{\top} g_k > \theta_k \|g_k \|^2$.
        \item  $\|d_k\| =  \|g_k\|$.
        \item Let \ref{B.2}--\ref{B.3} be satisfied, then we have $f(x_k) -f(x_{k+1}) \ge \sigma \|g_k\| \|s_k\| \cdot \min\{\cB,\frac{s\cZ}{2},s^{-\beta} b_k \|s_k\|^{\beta}, \frac{s^{1-\zeta}}{2} z_k \|s_k\|^{\zeta}\}$ where $s_k = \lambda_k d_k = \lambda_k \frac{\|g_k\|}{\|r_{t-1}\|}r_{t-1}$.
    \end{enumerate} 
 \end{lem}
 \begin{proof}
    In the case $\text{flag}_k = \NPC$, the inexactness criterion does not hold, i.e., we have $\|r_{t-1}\| > \theta_k\|g_k\|$. By \eqref{eq:bTrk} in Theorem~\ref{thm: NPC},  we then obtain $-d_k^{\top} g_k = \|r_{t-1}\|\|g_k\| > \theta_k \|g_k\|^2$. Part (ii) follows from the definition of $d_k$. Using the linesearch condition \eqref{armijo2}, it further holds that
\begin{align} \nonumber  \frac{f(x_k)-f(x_{k+1})}{\sigma \|g_k\|} & \geq -\lambda_k \frac{g_k^\top d_k}{\|g_k\|} - \frac{\lambda_k^2}{2} \frac{d_k^\top B_k d_k}{\|g_k\|} \\ & \geq \lambda_k \|r_{t-1}\| + \frac{\lambda_k^2}{2} \frac{\zeta_k\|d_k\|^2}{\|g_k\|} > \theta_k \|s_k\| + \frac{\lambda_k}{2}\zeta_k \|s_k\|. \label{eq:why-not?}
\end{align}
If $\cB \leq b_k \|g_k\|^\beta$, this implies $f(x_k) - f(x_{k+1}) \geq \cB\sigma \|g_k\| \|s_k\|$. Otherwise, if $\cB \geq b_k \|g_k\|^\beta$, we discuss two sub-cases. First, if $\lambda_k \le s$, we have
\[ {(f(x_k)-f(x_{k+1}))}/{(\sigma \|g_k\|)} \ge b_k \|g_k\|^\beta \|s_k\| = b_k \|d_k\|^\beta \|s_k\| \ge s^{-\beta} b_k\|s_k\|^{1+\beta}. \]
If $\lambda_k >s$, then, due to $\zeta \le 1$ in \ref{B.3}, we can infer
\begin{equation*}
    {2(f(x_k)-f(x_{k+1}))}/{(\sigma \|g_k\|)} \ge \lambda_k\zeta_k  \|s_k\| \ge \min\{s \cZ, s^{1-\zeta} z_k\|s_k\|^{\zeta}\} \|s_k\|. 
\end{equation*}
Combining the previous estimates, this finishes the proof.  
\end{proof}

\begin{lem}[Properties of GD-directions]
    \label[lemma]{lemma: LS_GD}
    Let $d_k=-g_k$ be generated by \Cref{alg:main_LS} (i.e., in Step 11 of \Cref{alg:main_LS}, we have $\mathrm{flag}_k = \mathrm{GD}$).  
    Then:
    \begin{enumerate}[label=\textup{(\roman*)},topsep=1ex,itemsep=0ex,partopsep=0ex, leftmargin = 25pt]
        \item $-d_k^{\top} g_k = \|g_k \|^2$.
        \item  $\|d_k\| =  \|g_k\|$.
        \item $f(x_k) -f(x_{k+1}) \ge \sigma \|g_k\| \|s_k\|$ (where $s_k = \lambda_k d_k = -\lambda_k g_k$).
    \end{enumerate} 
 \end{lem}

The properties listed in Lemma~\ref{lemma: LS_GD} follow immediately from $d_k = -g_k$.

In the classical results on truncated Newton methods, \cite{dembo1983truncated,dahito2019conjugate}, global convergence 
is shown under \ref{A.1}, \ref{B.1}, and the boundedness of the level set $\mathcal{L}_{f}(x_0):=\{x: f(x) \le f(x_0)\}$. Here, we use a different strategy. Firstly, we verify that every accumulation point of $\{x_k\}$ is a stationary point under \ref{A.1}, \ref{B.1}--\ref{B.3} or \ref{A.1}, \ref{B.4}, \ref{B.3}. Secondly, in \Cref{ssec: KL}, we prove convergence of the entire sequence $x_k\to x^*$ 
if the well-known Kurdyka-{\L}ojasiewicz property is satisfied at the accumulation point $x^*$. This geometric approach allows us to avoid additional boundedness assumptions. We now state our first global convergence result which is motivated by \cite[Theorem 3.1]{gould2000exploiting}.

\begin{thm}[Global convergence] \label[theorem]{global_LS}
Let the conditions \ref{A.1}, \ref{B.1}--\ref{B.3} or \ref{A.1}, \ref{B.4}, \ref{B.3} hold and let $\{x_k\}$ be generated by \Cref{alg:main_LS}. Then, every accumulation point of $\{x_k\}$ is a stationary point of \eqref{tar-func}.
\end{thm}

\begin{proof}
We first assume that the conditions \ref{A.1}, \ref{B.1}--\ref{B.3} are satisfied. 

Let $x^*$ be an accumulation point of $\{x_k\}$ and let $\{x_{k_{\ell}}\}$ be a subsequence converging to $x^*$. By the linesearch conditions \eqref{armijo} and \eqref{armijo2}, $\{f(x_k)\}$ is non-increasing and converges to some limit $\xi\in \bR \cup \{-\infty\}$. By the continuity of $f$, we can infer $f(x_{k_\ell}) \to f(x^*)$ and by the uniqueness of limits, this implies $\xi=f(x^*) \in \bR$. Thus, defining $\mathcal N_{\mathrm{npc}} := \{k: \mathrm{flag}_k = \NPC\}$, we have
\begin{equation}\label{eq:LSf}
    \begin{aligned}
        \infty &> f(x_{0})-f(x^*)
        =\lim _{i \rightarrow \infty} {\sum}_{k=0}^{i-1} f(x_{k})-f(x_{k+1}) \\
        &\geq -\sigma {\sum}_{k=0}^{\infty} \lambda_k d_k^{\top} g_k -\frac{\sigma}{2} {\sum}_{k\in \mathcal N_{\mathrm{npc}}} \lambda_k^2 d_k^{\top} B_k d_k \\
        &\geq -\sigma {\sum}_{k\notin \mathcal N_{\mathrm{npc}}} \lambda_k d_k^{\top} g_k + \frac{\sigma}{2}{\sum}_{k \in \mathcal N_{\mathrm{npc}}} \lambda_k^2 \zeta_k \|d_k\|^2.
    \end{aligned}
\end{equation} 
If $\|\nabla f(x^*)\| \neq 0$, then there exist $\epsilon >0$ and $L \in \mathbb{N}$ such that $\|g_{k_\ell}\| >\epsilon$ for $\ell \ge L$. Furthermore, by \ref{B.1}, $a_k \geq \bar a$, $\zeta_k \leq \cZ$, and parts~(i)--(ii) in Lemmas \ref{lemma: LS_SOL} and \ref{lemma: LS_GD}, there is $\bar \epsilon >0$ such that for all $k_\ell \notin \mathcal N_{\mathrm{npc}}$ and $\ell \ge L$, we have 
\begin{equation}\label{eq:dkgk/dk}
    -\frac{d_{k_\ell}^{\top} g_{k_\ell}}{\|d_{k_\ell}\|} \ge \frac{ \min\left\{ \frac{\min\{\cA, a_{k_\ell}\|g_{k_\ell}\|^{\alpha}\}}{\|\bar B_{k_\ell}\| + \|\bar B_{k_\ell}\|^2},1 \right\}  \|g_{k_\ell}\|^2}{\max\{\cA^{-1} \|g_{k_\ell}\|, a_{k_\ell}^{-1}\|g_{k_\ell}\|^{1-\alpha} ,\|g_{k_\ell}\| \}} > \bar \epsilon.
\end{equation}
Combining \eqref{eq:LSf}, \eqref{eq:dkgk/dk} and noting $\zeta_{k_\ell} \geq \min\{\cZ, \bar z \epsilon^\zeta\} =: \bar\zeta$ (for all $\ell \geq L$), this shows $\lambda_{k_\ell} \|d_{k_\ell}\| \to 0$ as $\ell \to \infty$.

We first consider the case where the Armijo condition \eqref{armijo} is used to determine the step size $\lambda_k$ for all $k$ sufficiently large, i.e., we assume $|\mathcal N_{\mathrm{npc}}| < \infty$. Since the step size $\rho^{-1} \lambda_{k_\ell}$ does not satisfy \eqref{armijo}, we have
\begin{equation}\label{eq:noarmijo}
    f (x_{k_\ell}+ \rho^{-1} \lambda_{k_\ell} d_{k_\ell})-f(x_{k_\ell}) >\sigma \rho^{-1} \lambda_{k_\ell} g_{k_\ell}^{\top} d_{k_\ell}.
\end{equation}
 In addition, by the mean-value theorem, there is $\tau_{k_\ell} \in [0,1]$ such that
\begin{equation}\label{eq:taylor}
    f(x_{k_\ell}+ \rho^{-1} \lambda_{k_\ell} d_{k_\ell})  -f(x_{k_\ell}) = \rho^{-1}\lambda_{k_\ell} \nabla f(x_{k_\ell}+ \tau_{k_\ell}\rho^{-1} \lambda_{k_\ell} d_{k_\ell})^{\top} d_{k_\ell}.
\end{equation}
Setting $\hat d_{k_\ell} = \tau_{k_{\ell}} \rho^{-1} \lambda_{k_{\ell}} d_{k_{\ell}}$ and combining \eqref{eq:noarmijo} and \eqref{eq:taylor}, we obtain
\[ \frac{\sigma d_{k_{\ell}}^{\top} g_{k_{\ell}}}{\|d_{k_{\ell}}\|} < \frac{\nabla f(x_{k_{\ell}}+\hat d_{k_\ell})^{\top} d_{k_{\ell}}}{\|d_{k_{\ell}}\|} \leq\|\nabla f(x_{k_{\ell}}+\hat d_{k_\ell})-\nabla f(x_{k_{\ell}})\|+\frac{d_{k_{\ell}}^{\top} g_{k_{\ell}}}{\|d_{k_{\ell}}\|}.
\]
Since $\{x_{k_{\ell}}\}$ converges to $x^*$ and $\nabla f$ is uniformly continuous on compact sets, there exist $r >0$ and $L'>L$ such that 
$$\|\nabla f(x_{k_\ell} + d) - \nabla f(x_{k_\ell})\| \le (1- \sigma) \bar \epsilon, \quad \forall~d\in B_{r}(0),\quad \forall~\ell >L'. $$
Due to $\lambda_{k_\ell} \|d_{k_\ell}\| \to 0$, it holds that $\|\hat d_{k_\ell}\| \to 0$ as $\ell \to \infty$ and together with \eqref{eq:dkgk/dk}, this yields the contradiction
$ (1-\sigma)\bar \epsilon > -(1-\sigma){d_{k_\ell}^{\top} g_{k_\ell}}/{\|d_{k_\ell}\|} > (1-\sigma) \bar \epsilon$ for all $\ell$ sufficiently large.

Next, we consider the case where the alternative linesearch condition \eqref{armijo2} is used. We only need to discuss the situation when the subsequence $\{k_\ell\}$ contains infinitely many elements for which \eqref{armijo2} is applied, i.e., $|\{k_\ell\} \cap \mathcal N_{\mathrm{npc}}| = \infty$. As before, for all such $k_\ell$, since the step size $\rho^{-1} \lambda_{k_\ell}$ does not satisfy \eqref{armijo2}, we have 
\begin{equation}\label{eq:noarmijo2}
    f(x_{k_\ell}+ \rho^{-1} \lambda_{k_\ell} d_{k_\ell})-f(x_{k_\ell}) >\sigma (\rho^{-1} \lambda_{k_\ell} g_{k_\ell}^{\top} d_{k_\ell} +  \rho^{-2} \lambda_{k_\ell}^2 d_{k_\ell}^\top B_{k_\ell} d_{k_\ell}/2).
\end{equation}
In addition, by the mean-value theorem, there is again $\tau_{k_\ell} \in [0,1]$ such that
\begin{equation} \label{eq:taylor2}
    f(x_{k_{\ell}}+\rho^{-1} \lambda_{k_{\ell}} d_{k_{\ell}})-f(x_{k_{\ell}}) = \rho^{-1} \lambda_{k_{\ell}} g_{k_\ell}^{\top} d_{k_\ell} + \rho^{-2} \lambda_{k_{\ell}}^2 d_{k_\ell}^\top \hat B_{k_\ell} d_{k_\ell}/2,
\end{equation}
where $\hat B_{k_\ell} = \nabla^2 f(x_{k_{\ell}}+\hat d_{k_\ell})$ and $\hat d_{k_\ell} = \tau_{k_{\ell}} \rho^{-1} \lambda_{k_{\ell}} d_{k_{\ell}}$. Hence, combining \eqref{eq:noarmijo2}, \eqref{eq:taylor2}, Lemma~\ref{lemma: LS_NC} (i), $\|g_{k_\ell}\| \geq \epsilon$, $\theta_{k_\ell} \geq \min\{\cB,\bar b \epsilon^\beta\} =: \bar\theta$, and $d_{k_\ell}^\top B_{k_\ell} d_{k_\ell} \le -\zeta_{k_\ell} \|d_{k_\ell} \|^2 \le 0$, this yields
\begin{equation}\label{eq:above1}
    \begin{aligned}
        &\rho^{-1} \lambda_{k_{\ell}} d_{k_\ell}^\top (\hat B_{k_\ell}-B_{k_\ell}) d_{k_\ell}/2 > (\sigma-1)(g_{k_\ell}^{\top} d_{k_\ell} + \rho^{-1} \lambda_{k_{\ell}} d_{k_\ell}^\top B_{k_\ell} d_{k_\ell}/2) \\
        &\ge (\sigma-1) g_{k_\ell}^{\top} d_{k_\ell} > (1-\sigma) \theta_{k_\ell} \|d_{k_\ell}\| \|g_{k_\ell}\| > (1-\sigma) \bar\theta \epsilon \|d_{k_\ell}\|, \quad \ell \ge L.
    \end{aligned}
\end{equation}
Dividing \eqref{eq:above1} by $\|d_{k_\ell}\|$ and using $\lambda_{k_\ell} \|d_{k_\ell}\|, \|\hat d_{k_\ell}\| \to 0$, $\ell \to \infty$ and the bound- edness of $B_{k_\ell}$ and $\hat B_{k_\ell}$, we obtain a contradiction. This proves $\nabla f(x^*) = 0$. 

Finally, let the assumptions \ref{A.1}, \ref{B.4}, \ref{B.3} hold. In this case, our previous steps are still applicable. In particular, the bounds \eqref{eq:LSf} and \eqref{eq:dkgk/dk} remain valid, since $\|B_{k_\ell}\| = \|\nabla^2 f(x_{k_\ell})\|$ is bounded for all $\ell \ge L$. Thus, we still have $\lambda_{k_\ell} \|d_{k_\ell}\| \to 0$ as $\ell \to \infty$. Since the derivations in the case $|\mathcal N_{\mathrm{npc}}| < \infty$ do not depend on \ref{B.1} and \ref{B.2}, we only need to prove a contradiction when \eqref{armijo2} is used. Combining \eqref{eq:noarmijo2} and \eqref{eq:taylor2} and applying $d_{k_\ell}^\top B_{k_\ell} d_{k_\ell} \le -\zeta_{k_\ell} \|d_{k_\ell} \|^2 \le -\bar\zeta\|d_{k_\ell}\|^2$, it follows
        \begin{equation}\label{eq:above1new}
            \begin{aligned}
                &\rho^{-1} \lambda_{k_{\ell}} d_{k_\ell}^\top (\hat B_{k_\ell}-B_{k_\ell}) d_{k_\ell}/2 > (\sigma-1)(g_{k_\ell}^{\top} d_{k_\ell} + \rho^{-1} \lambda_{k_{\ell}} d_{k_\ell}^\top B_{k_\ell} d_{k_\ell}/2) \\
                &\ge (\sigma-1) \rho^{-1} \lambda_{k_{\ell}} d_{k_\ell}^\top B_{k_\ell} d_{k_\ell}/2 > (1-\sigma) \rho^{-1} \lambda_{k_{\ell}} \bar \zeta \|d_{k_\ell}\|^2/2.
            \end{aligned}
        \end{equation}
        Dividing \eqref{eq:above1new} by $\rho^{-1} \lambda_{k_{\ell}} \|d_{k_\ell}\|^2$, using $\|\hat d_{k_\ell}\| \to 0$, $\ell \to \infty$ and the continuity of $\nabla^2 f$, we again obtain a contradiction. 
\end{proof}

\subsection{Convergence under the Kurdyka-{\L}ojasiewicz Inequality} \label{ssec: KL}

We now provide additional convergence results for \Cref{alg:main_LS} invoking the celebrated Kurdyka-{\L}ojasiewicz (KL) inequality. The KL property characterizes the growth behavior of the function $f$ locally around critical points and can be used to establish stronger convergence guarantees for the iterates.

By $\mathfrak L_\eta$, we denote the class of all continuous and concave desingularizing functions $\varrho: [0, \eta) \to \bR_+$ such that
\begin{equation*} 
\varrho \in C^1((0, \eta)), \quad \varrho(0)=0, \quad \varrho'(x)>0, \quad \forall~x\in (0,\eta).
\end{equation*}
In the following, we state a definition of the KL property for $f$, cf. \cite{absil2005convergence,AttBol09}. 

\begin{defn} \label[definition]{def:KL}
We say that $f$ has the Kurdyka-{\L}ojasiewicz property at $\bar x$ if there exist $\eta \in (0,\infty]$, a neighborhood $V$ of $\bar x$, and a function $\varrho \in \mathfrak L_\eta$ such that for all $x \in V\cap \{x \in \bR^n : 0 < f(x) - f(\bar x) < \eta\}$ the KL-inequality holds, i.e.,
\begin{equation}\label{eq:KL}
        \varrho^\prime(f(x)-f(\bar x)) \|\nabla f(x)\| \ge 1.
    \end{equation}
If the mapping $\varrho$ further satisfies $\varrho(t) = ct^{1-\theta}$ for some $c > 0$
and $\theta \in [0, 1)$, then we say that $f$ has the {\L}ojasiewicz property at $\bar x$ with exponent $\theta$.
\end{defn}

KL-based analysis techniques have been highly successful and enjoy wide applicability in practice, \cite{absil2005convergence,AttBol09,AttBolSva13,BolSabTeb14}. The KL and {\L}ojasiewicz property introduced in \Cref{def:KL} hold for the rich and ubiquitous classes of tame, subanalytic, and semialgebraic functions, see, e.g., \cite{lojasiewicz1963,lojasiewicz1993,kurdyka1998,BolDanLew06}. 

Let $\{x_k\}$ be generated by \Cref{alg:main_LS} and let us define the set of accumulation points $\mathcal A := \{x^*: \liminf_{k\to\infty}\|x_k-x^*\| = 0\}$. We consider the following key assumptions.

\begin{mdframed}[style=assumptionbox]
\begin{enumerate}[label=\textup{\textrm{(A.\arabic*)}},topsep=2pt,itemsep=0ex,partopsep=0ex,leftmargin=6ex]
        \setcounter{enumi}{1}
        \item \label{A.2} $\nabla^2 f$ is locally Lipschitz continuous on a neighborhood $V_H$ of $\mathcal A$.
\end{enumerate}
\begin{enumerate}[label=\textup{\textrm{(C.\arabic*)}},topsep=2pt,itemsep=0.5ex,partopsep=0ex,leftmargin=6ex]
    \setcounter{enumi}{0}
    \item \label{C.1} The parameter sequences $\{a_k\}$ and $\{z_k\}$ satisfy ${\sum}_{k=1}^\infty a_k^{-{1}/{\alpha}} < \infty$ and ${\sum}_{k=1}^\infty z_k^{-{1}/{\zeta}} < \infty$.
    \item \label{C.2} $f$ satisfies the KL property, \eqref{eq:KL}, at some accumulation point $x^* \in \mathcal A$. 
    
    \item \label{C.3} The function $f$ satisfies the {\L}ojasiewicz property at some accumulation point $x^* \in \mathcal A$ with exponent $\theta \in [0,1)$. 
\end{enumerate} 
\end{mdframed}

\ref{C.1} specifies the growth behavior of the parameters $\{a_k\}$ and $\{z_k\}$. A simple exemplary choice is given by $a_k = (k \log(k)^2)^{\alpha}$ and $z_k = (k \log(k)^2)^{\zeta}$. Based on \ref{C.2}, we show that the whole sequence of iterates $\{x_k\}$, generated by \Cref{alg:main_LS}, converges to a critical point $x^*$. Our proof relies on the descent estimates in \Cref{lemma: LS_SOL,lemma: LS_NC,lemma: LS_GD} (iii) and adapts the original analysis in \cite{absil2005convergence} and the mathematical induction technique recently utilized in \cite{jia2023convergence} to our MINRES-based framework. 

Based on assumption \ref{A.2}, we now first derive lower bounds on the stepsizes when the iterates are close to an accumulation point.

\begin{lem}\label[lemma]{lemma:stepsize}
Let $\{x_k\}$ be generated by \Cref{alg:main_LS} and let \ref{A.1}--\ref{A.2} hold. Let $r > 0$ and $x^* \in \mathcal A$ be given such that $B_{2r}(x^*) \subset V_H$ and let $L_g \equiv L_g(x^*)$ and $L_H \equiv L_H(x^*)$ denote the Lipschitz constants of $\nabla f$ and $\nabla^2 f$ on $B_{2r}(x^*)$, respectively. Suppose further that $x_k, x_{k+1} \in B_{2r}(x^*)$. Then, it holds that:
    \begin{enumerate}[label=\textup{(\roman*)},topsep=1ex,itemsep=0ex,partopsep=0ex, leftmargin = 25pt]
        \item $\lambda_k \ge \min\{s, \frac{2\rho(1-\sigma)}{L_{g}}\min\{\cA, a_k \|g_k\|^\alpha\}\}$ if $\mathrm{flag}_k \neq \mathrm{NPC}$. 
        \item   $\lambda_k \ge \frac{ 3\rho(\sigma-1) d_k^{\top} B_k d_k }{ L_H \|d_k\|^3}$ if $\mathrm{flag}_k=\mathrm{NPC}$ and $B_k = \nabla^2 f(x_k)$.
    \end{enumerate}
\end{lem}

\begin{proof}
By the descent lemma, it follows
\begin{align} \label{eq:lipschitz}
f(x_{k+1}) - f(x_k) - \sigma \lambda_k d_k^\top g_k \leq (1-\sigma) \lambda_k d_k^\top g_k + \frac{L_g}{2}\lambda_k^2 \|d_k\|^2. 
\end{align}
Thus, \eqref{armijo} is satisfied for all $\lambda_k \in (0,\mu^1_k]$ where $\mu^1_k := -2(1-\sigma)\frac{d_k^{\top}g_k}{L_g\|d_k\|^2}$ and we have $\lambda_k \geq \rho \mu^1_k$. Specifically, for $d_k=p_t$, applying \Cref{lemmaMR}~(i) and \eqref{eq:cur}, it follows
\[ \lambda_{k}  \geq \frac{2\rho(1-\sigma) p_t^{\top} \bar B_k p_t}{L_g\|p_{t}\|^{2}}  \ge  \frac{2\rho(1-\sigma) }{L_g} \min\{\cA, a_k \|g_k\|^{\alpha} \}. \]
For $d_k=-g_{k}$, using $-d_k^{\top}g_k = \|d_k\|^2$, we obtain $\lambda_k \ge 2\rho(1-\sigma)/L_g$. Noting $\cA \leq 1$, this proves (i). Next, applying the Lipschitz continuity of $\nabla^2 f$ (on $B_{2r}(x^*)$) and $g_k^\top d_k \leq 0$, we have  
\begin{align*}
   & f(x_{k+1}) - f(x_k) - \sigma (\lambda_k g_k^{\top} d_k + \lambda_k^2  d_k^{\top} B_k d_k/2 ) \\ & \hspace{6ex} \leq (1-\sigma)\lambda_k (g_k^{\top} d_k + \lambda_k  d_k^{\top} B_k d_k/2 ) + {L_H}\lambda_k^3 \|d_k\|^3/6 \\ & \hspace{6ex} \leq 0.5{\lambda_k^2}((1-\sigma)d_k^\top B_kd_k + L_H\lambda_k\|d_k\|^3/3).   
\end{align*}
Hence, \eqref{armijo2} holds for all $\lambda_k \in (0,\mu_k^2]$ where $\mu_k^2 := 3(\sigma-1)d_k^\top B_kd_k / (L_H\|d_k\|^3)$ and, by the linesearch mechanism for $\mathrm{NPC}$-directions, (cf.\ Algorithm~\ref{alg: LS}), we can infer $\lambda_k \geq \rho\mu_k^2$. This finishes the proof.
\end{proof}

In the following and based on the lower bounds in \Cref{lemma:stepsize}, we present one of our main convergence results for \Cref{alg:main_LS}.

\begin{thm}[Convergence under the KL property, I] \label[theorem]{global_LS_KL_new}
 Let \ref{A.1}--\ref{A.2}, \ref{B.4}, \ref{B.3}, and \ref{C.1} hold, and let $\{x_k\}$ be generated by \Cref{alg:main_LS}. Suppose that $x^* \in \mathcal A$ is an accumulation point of $\{x_k\}$ at which assumption \ref{C.2} is satisfied. Then either the algorithm stops after finitely many iterations at a stationary point or the whole sequence $\{x_k\}$ converges to $x^*$ with $\|g_k\| \to 0$.
\end{thm}

\begin{proof} We may assume that the algorithm does not stop after finitely many steps. Then, by \Cref{global_LS}, $x^*$ is a stationary point of $f$. Due to $f(x_k) \searrow f(x^*)$, $k \to \infty$ (as shown in \eqref{eq:LSf}), there exists $k_\eta \in \bN$ such that for all $k \ge k_\eta$, we have $0 < f(x_k)-f(x^*) \le \eta$. 
    
We first focus on the case $\alpha<1$ and recall $\mathcal N_{\mathrm{npc}} := \{k: \mathrm{flag}_k = \NPC\}$. By \eqref{eq:why-not?}, we have 
\[ f(x_k)-f(x_{k+1}) \geq \sigma\|g_k\|\|s_k\| \cdot {\lambda_k\zeta_k}/{2}\]
if $\mathrm{flag}_k = \NPC$. (Notice that we do not assume condition \ref{B.2} to hold and hence, \Cref{lemma: LS_NC} (iii) is not applicable). Combining Lemma \ref{lemma: LS_SOL}~(iii), Lemma \ref{lemma: LS_GD}~(iii), and \eqref{eq:why-not?}, it follows 
    \begin{equation}\label{eq:KL_allnew}
        f(x_k) -f(x_{k+1}) \ge \tau_k \sigma \|g_k\| \|s_k\|, \quad \tau_k := \begin{cases} \min\{ T_{1,k}, T_{2,k} \} & \text{if $k \notin \mathcal N_{\mathrm{npc}}$}, \\ \min\{T_{3,k},T_{4,k} \} & \text{if $k \in \mathcal N_{\mathrm{npc}}$}, \end{cases}
    \end{equation}
where $T_{1,k} := \min\{ C_k\cA^2, 1\}$, $T_{2,k} := C_k (a_k/s^{\alpha})^\frac{2}{1-\alpha}\|s_k\|^{\frac{2\alpha}{1-\alpha}}$, $T_{3,k} := \lambda_k  \cZ/2$ and $T_{4,k} := \lambda_k^{1-\zeta}z_k\|s_k\|^{\zeta}/2$. (The expression for $T_{4,k}$ is due to $\|g_k\|^\zeta = \|d_k\|^\zeta = \lambda_k^{-\zeta}\|s_k\|^\zeta$, cf. \Cref{lemma: LS_NC} (ii)). One of the main steps of the proof is to ensure that $\{\lambda_k\}$ (in $ T_{3,k}$ and $T_{4,k}$) remains bounded from below.

Let $\{x_{k_\ell}\}$ be a subsequence converging to $x^*$. Furthermore, let $r>0$ be given such that $B_{2r}(x^*) \subset V_H$ and let $L_g$, $L_H$ denote the Lipschitz constants of $\nabla f$ and $\nabla^2 f$ on $B_{2r}(x^*)$.
Using the continuity of $\nabla f$, $\zeta_k = \min\{\cZ,z_k\|g_k\|^\zeta\}$, $\zeta \leq 1$, and \Cref{lemma: LS_NC} (ii), there exists $\delta>0$ such that for all $x_k \in B_{\delta}(x^*)$ and $k \in \mathcal N_{\mathrm{npc}}$, it holds that $\|g_k\| \le \min\{{r}/{s},1\}$ and 
\begin{equation} \label{eq:use-later} \frac{3\rho(\sigma-1) d_k^{\top} B_k d_k }{L_H \|d_k\|^3} \ge \frac{3\rho(1-\sigma) \zeta_k }{L_H \|d_k\|} \ge \frac{3\rho(1-\sigma)}{L_H} \min\Big\{ \frac{\cZ}{ \|g_k\|}, \frac{\bar z}{ \|g_k\|^{1-\zeta}}\Big\} \ge  \bar\tau, \end{equation}
where $\bar \tau := 3\rho(1-\sigma)\min\{\cZ,\bar z\}/L_H$. Let us define
\begin{equation} \label{eq:defsigma-new}
       \sigma_{1,k} := {\sum}_{i=k}^\infty a_i^{-\frac{1}{\alpha}}, \quad \sigma_{2,k} := {\sum}_{i=k}^\infty z_i^{-\frac{1}{\zeta}}, \quad \varrho_{k} := \varrho(f(x_{k}) - f(x^*)).
\end{equation}

    Next, there is $k_{\ell_0} \ge k_\eta$ such that $\nu \le \min\{r, \delta\}$ and $B_\nu (x^*)  \subset V$, where
    $$ \nu := \|x_{k_{\ell_0}}-x^*\| + \frac{\tilde\varrho}{\tilde T\sigma} + (s\tilde\sigma_1)^{1-p_1}\Big(\frac{\tilde\varrho}{\tilde C \sigma}\Big)^{p_1} +  \tilde\sigma_2^{1-p_2}\Big(\frac{\tilde\varrho}{\tilde T_2 \sigma}\Big)^{p_2}, $$
    $\tilde\varrho := \varrho_{k_{\ell_0}}$, $\tilde\sigma_{i} := \sigma_{i,k_{\ell_0}}$, $i=1,2$, $p_1 := \frac{1-\alpha}{1+\alpha}$, $p_2 := \frac{1}{1+\zeta}$, $\tilde C := [(L_g+\cZ)(L_g+\cZ+1)]^{-1}$, 
    $\tilde T = \min\{\tilde C\cA^2, 1,\bar\lambda\cZ/2\}/2$, $\tilde T_2:= \bar \lambda^{1-\zeta}/2$, 
    $\bar \lambda:= \min\{s,\bar \tau\}$, and the set $V$ is defined in \ref{C.2}. The existence of such index $k_{\ell_0}$ follows from $x_{k_\ell} \to x^*$, $\varrho_{k_\ell}\to 0$ and $\sigma_{i,k_\ell} \to 0$, $i=1,2$, as $\ell \to \infty$. We now show by induction that the following three statements hold for all $k \ge k_{\ell_0}$:
    \begin{equation}\label{eq:induction-new}
        x_k \in B_\nu (x^*), \quad \lambda_k \ge \bar \lambda\;\text{if $k \in \mathcal N_{\mathrm{npc}}$}, \quad {\sum}_{i=k_{\ell_0}}^{k} \|s_i\|\le \nu - \|x_{k_{\ell_0}}-x^*\|. 
    \end{equation}
    Per definition of $\nu$, it holds that $x_{k_{\ell_0}} \in B_\nu (x^*)$. Suppose that we have $\lambda_{k_{\ell_0}} < \bar \lambda$ and $k_{\ell_0} \in \mathcal N_{\mathrm{npc}}$. Then, due to $x_{k_{\ell_0}} \in B_\nu(x^*) \subset B_\delta(x^*)$, $\|d_{k_{\ell_0}}\|=\|g_{k_{\ell_0}}\|$, and $\bar\lambda\leq s$, we obtain $\|x_{k_{\ell_0}+1}-x_{k_{\ell_0}}\| = \lambda_{k_{\ell_0}} \|d_{k_{\ell_0}}\| \le s  \|g_{k_{\ell_0}}\| \le r$. Using $x_{k_{\ell_0}} \in B_\nu(x^*) \subset B_r(x^*)$, this implies $x_{k_{\ell_0}+1} \in B_{2r}(x^*)$. Thus, part (ii) in \Cref{lemma:stepsize} is applicable and invoking $x_{k_{\ell_0}} \in B_\delta(x^*)$ and \eqref{eq:use-later}, it follows
    \[ \lambda_{k_{\ell_0}} \ge  \frac{ 3\rho(\sigma-1) d_{k_{\ell_0}}^{\top} B_{k_{\ell_0}} d_{k_{\ell_0}} }{ L_H \|d_{k_{\ell_0}}\|^3} \ge  \min\{s,\bar \tau\} = \bar \lambda. \]
    This is a contradiction and we either need to have $\lambda_{k_{\ell_0}} \ge \bar\lambda$ and $k_{\ell_0} \in \mathcal N_{\mathrm{npc}}$ or $k_{\ell_0} \notin \mathcal N_{\mathrm{npc}}$. Noticing $B_\nu(x^*) \subset B_{2r}(x^*)$, $\max_{x \in B_{2r}(x^*)} \|\nabla^2 f(x)\| \leq L_g$, and $\zeta_k \leq \cZ$, we have $C_{k_{\ell_0}} \ge \tilde C$ by the definition of $C_k$ and $\tilde C$. Moreover, by \eqref{eq:KL}, \eqref{eq:KL_allnew}, $x_{k_{\ell_0}} \in V$, and the concavity and nonnegativity of $\varrho$, it holds that
    \begin{align} \nonumber
    \tau_{k_{\ell_0}}\sigma \|s_{k_{\ell_0}}\| &\le \frac{f(x_{k_{\ell_0}}) - f(x_{k_{\ell_0}+1})}{\|g_{k_{\ell_0}}\|} \\ & \hspace{-12ex} \le  \varrho^\prime(f(x_{k_{\ell_0}})-f(x^*))({f(x_{k_{\ell_0}}) - f(x_{k_{\ell_0}+1})}) 
          \le \varrho_{k_{\ell_0}} - \varrho_{k_{\ell_0}+1} \le \tilde\varrho. \label{eq:use-this-one-new}
    \end{align}
    Finally, let us define ${\mathcal K}_1 := \{k \ge k_{\ell_0}, k \notin \mathcal N_{\mathrm{npc}}: T_{1,k} \leq T_{2,k}\}$, ${\mathcal K}_2 := \{k \ge k_{\ell_0}, k \notin \mathcal N_{\mathrm{npc}}: T_{1,k} > T_{2,k}\}$, ${\mathcal K}_3 := \{k \ge k_{\ell_0}, k \in \mathcal N_{\mathrm{npc}}: T_{3,k} \leq T_{4,k}\}$, and ${\mathcal K}_4 := \{k \ge k_{\ell_0}, k \in \mathcal N_{\mathrm{npc}}: T_{3,k} > T_{4,k}\}$. Invoking \eqref{eq:KL_allnew}, \eqref{eq:use-this-one-new}, and $C_{k_{\ell_0}} \ge \tilde C$, the condition $T_{2,k_{\ell_0}}\|s_{k_{\ell_0}}\| \leq \tilde\varrho/\sigma$ implies 
    \[ \|s_{k_{\ell_0}}\| \leq s^{\frac{2\alpha}{1+\alpha}} a_{k_{\ell_0}}^{-\frac{2}{1+\alpha}} ({\tilde\varrho}/{(\tilde C\sigma)})^{\frac{1-\alpha}{1+\alpha}} \leq (s \tilde\sigma_1)^{\frac{2\alpha}{1+\alpha}} ({\tilde\varrho}/{(\tilde C\sigma)})^{\frac{1-\alpha}{1+\alpha}}. \]
    %
    %
    Similarly, in the cases $k_{\ell_0} \notin \mathcal K_2$, we can obtain the following 
    %
   \[ 
    \begin{cases}
     T_{1,k_{\ell_0}}\|s_{k_{\ell_0}}\| \leq \frac{\tilde\varrho}{\sigma} \implies \|s_{k_{\ell_0}}\| \leq \frac{\tilde\varrho}{2\tilde T\sigma} & \text{if $k_{\ell_0} \in {\mathcal K}_1$}, \\
        T_{3,k_{\ell_0}}\|s_{k_{\ell_0}}\| \leq \frac{\tilde\varrho}{\sigma} \implies \|s_{k_{\ell_0}}\| \leq \frac{2\tilde\varrho}{\cZ\bar\lambda\sigma} \leq \frac{\tilde\varrho}{2\tilde T\sigma} & \text{if $k_{\ell_0} \in {\mathcal K}_3$}, \\
        T_{4,k_{\ell_0}}\|s_{k_{\ell_0}}\| \leq \frac{\tilde\varrho}{\sigma} \implies \|s_{k_{\ell_0}}\| \le \tilde\sigma_2^{\frac{\zeta}{1+\zeta}} (\frac{\tilde\varrho}{\tilde T_2 \sigma})^{\frac{1}{1+\zeta}}  & \text{if $k_{\ell_0} \in {\mathcal K}_4$},
    \end{cases} \]
    where we used $\lambda_{k_{\ell_0}} \ge \bar \lambda$ if $k_{\ell_0} \in {\mathcal K}_3 \cup {\mathcal K}_4$. Thus, \eqref{eq:induction-new} holds for $k=k_{\ell_0}$. 

    Suppose \eqref{eq:induction-new} is satisfied for some $k > k_{\ell_0}$ and all $ k_{\ell_0} \le i \le k$. By the triangle inequality and the definition of $\nu$, we then have 
    $$\|x_{k+1} -x^*\| \le \|x_{k_{\ell_0}} -x^*\| + {\sum}_{i=k_{\ell_0}}^{k}\|s_i\| \le \nu. $$
    This proves $x_{k+1} \in B_\nu (x^*)$. The condition $\lambda_{k+1} \geq \bar\lambda$ if $k+1\in\mathcal N_{\mathrm{npc}}$ can be verified again via a contradiction and is identical to the base case $k = k_{\ell_0}$. 
    %
    %
    Applying the estimate \eqref{eq:use-this-one-new} for the iterates $x_i$, $i = k_{\ell_0},\dots,k+1$ (which is possible due to $x_{i} \in B_\nu (x^*)$ for all $i = k_{\ell_0},\dots,k+1$), we first note that
    \begin{equation}
    {\sum}_{i=k_{\ell_0}}^{k+1} \tau_i \|s_i\| \leq {\sum}_{i=k_{\ell_0}}^{k+1 } \frac{f(x_{i}) - f(x_{i+1})}{\sigma \|g_{i}\|} \le 
           {\sum}_{i=k_{\ell_0}}^{k+1} \frac{\varrho_{i} - \varrho_{i+1}}{\sigma}  
           \le \frac{\tilde \varrho}{\sigma}. \label{eq:sum1-new}
    \end{equation}
    By the inverse H\"older inequality and recalling $p_1 = \frac{1-\alpha}{1+\alpha} < 1$,  we further have
    \begin{align*}
        {\sum}_{i=k_{\ell_0}, i \in \mathcal K_2}^{k+1} \tau_i \|s_i\| & \geq \tilde C s^{-\frac{2\alpha}{1-\alpha}}{\sum}_{i=k_{\ell_0}, i \in \mathcal K_2}^{k+1} a_i^{\frac{2}{1-\alpha}} \|s_i\|^{\frac{1+\alpha}{1-\alpha}} \nonumber \\
            & \hspace{-12ex} \ge \tilde C s^{-\frac{2\alpha}{1-\alpha}} \Big({\sum}_{i=k_{\ell_0}, i \in \mathcal K_2}^{k+1} (a_i^{\frac{2}{1-\alpha}})^{\frac{-1}{\frac{1}{p_1}-1}} \Big)^{{1-\frac{1}{p_1}}} \Big({\sum}_{i=k_{\ell_0}, i \in \mathcal K_2}^{k+1} \|s_i\| \Big)^{\frac{1}{p_1}} \nonumber \\
            & \hspace{-12ex} = \tilde C s^{1-\frac{1}{p_1}} \Big({\sum}_{i=k_{\ell_0}, i \in \mathcal K_2}^{k+1} a_i^{-\frac{1}{\alpha}} \Big)^{1-\frac{1}{p_1}}\Big({\sum}_{i=k_{\ell_0}, i \in \mathcal K_2}^{k+1} \|s_i\| \Big)^{\frac{1}{p_1}} \nonumber \\
            & \hspace{-12ex} \geq \tilde C (s\tilde\sigma_1)^{1-\frac{1}{p_1}} \Big({\sum}_{i=k_{\ell_0}, i \in \mathcal K_2}^{k+1} \|s_i\| \Big)^{\frac{1}{p_1}}. 
        \end{align*}
    Here, the estimate $C_i \geq \tilde C$ is valid due to $x_i \in B_\nu(x^*) \subset B_{2r}(x^*)$ for all $i=k_{\ell_0},\dots,k+1$. Combining this with \eqref{eq:sum1-new}, we can infer ${\sum}_{i=k_{\ell_0}, i \in \mathcal K_2}^{k+1} \|s_i\| \leq (s\tilde\sigma_1)^{1-p_1}(\tilde\varrho/(\tilde C\sigma))^{p_1}$. Mimicking these steps and invoking the induction hypothesis, we can further obtain
    \[ {\sum}_{i=k_{\ell_0}, i \in \mathcal K_1}^{k+1 } \|s_i\|  \le  \frac{\tilde\varrho}{2\tilde T\sigma} , \quad {\sum}_{i=k_{\ell_0}, i \in \mathcal K_3}^{k+1 } \|s_i\|  \le \frac{2\tilde\varrho}{\sigma \bar\lambda \cZ} \leq \frac{\tilde\varrho}{2\tilde T\sigma}, \] 
    and $\sum_{i=k_{\ell_0}, i \in \mathcal K_4}^{k+1} \|s_i\| \le \tilde\sigma_2^{1-p_2}({\tilde\varrho}/{(\tilde T_2\sigma)})^{p_2}$. This shows that \eqref{eq:induction-new} is also satisfied for $k+1$ and completes the induction. Since \eqref{eq:induction-new} holds for all $k \ge k_{\ell_0}$, taking $k \to \infty$, we can deduce that $\{x_k\}$ is a Cauchy sequence and hence, it follows $x_k \to x^*$ and $\lim_{k\to \infty} \|g_k\| = \|\nabla f(x^*)\| = 0$.

    Finally, consider the special case $\alpha=1$. In this case, Lemma \ref{lemma: LS_SOL}~(iii) is not applicable if $\cA \ge a_k\|g_k\|$. However, if $\mathrm{flag}_k = \mathrm{SOL}$ and $\cA \ge a_k\|g_k\|$, then \Cref{lemma: LS_SOL}~(ii) implies $1 \ge a_k \|d_k\|$. Thus, redefining the index sets $\mathcal K_2 := \{k \ge k_{\ell_0}: \mathrm{flag}_k = \mathrm{SOL}, \cA \ge a_k\|g_k\|\}$ and $\mathcal K_1 := \{k \ge k_{\ell_0}: k \notin \mathcal N_{\mathrm{npc}}, k \notin \mathcal K_2\}$, it follows that $\bigcup_{i=1}^4 \mathcal K_i \supset \{k \ge k_{\ell_0}\}$ and ${\sum}_{i=k_{\ell_0}, i \in \mathcal K_2}^{\infty} \|s_i\| \le {\sum}_{i=k_{\ell_0}, i \in \mathcal K_2}^{\infty} \lambda_i/a_i \le s \tilde \sigma_1$. Consequently, noting $p_1 = 0$, the previous steps in the case $\alpha < 1$ are still all valid using the same definition of $\nu$\footnote{In the case $\mathrm{flag}_k = \mathrm{SOL}$ and $\cA < a_k\|g_k\|$, we have $f(x_k)-f(x_{k+1}) \geq C_k\cA^2 \sigma \|g_k\|\|s_k\|$, as shown in the proof of \Cref{lemma: LS_SOL}. Hence, using \Cref{lemma: LS_GD}~(iii), this ensures that \eqref{eq:KL_allnew} holds with $\tau_k = T_{1,k}$ for all $k\in\mathcal K_1$.}. This completes the proof.
\end{proof}

The results derived in \Cref{global_LS_KL_new} can be shown to remain valid if conditions \ref{A.2} and \ref{B.4} are replaced by \ref{B.1}--\ref{B.2}, i.e., if the matrices $\{B_k\}$ remain bounded and a specific format of the tolerance parameters $\{\theta_k\}$ is used. We now state this variant of \Cref{global_LS_KL_new}.

\begin{thm}[Convergence under the KL property, II] \label[theorem]{global_LS_KL}
 Let \ref{A.1}, \ref{B.1}--\ref{B.3}, and \ref{C.1} hold and let $\{x_k\}$ be generated by \Cref{alg:main_LS}. Let us further assume ${\sum}_{k=1}^\infty b_k^{-{1}/{\beta}} < \infty$ and suppose that $x^* \in \mathcal A$ is an accumulation point of $\{x_k\}$ at which assumption \ref{C.2} is satisfied. Then either the algorithm stops after finitely many iterations
at a stationary point or the whole sequence
$\{x_k\}$ converges to $x^*$ with $\lim_{k\to \infty}
\|g_k\| = 0$.
\end{thm}

\begin{proof}
The proof of \Cref{global_LS_KL} is similar to the derivation of \Cref{global_LS_KL_new}. In particular, Lemma \ref{lemma: LS_NC}~(iii) is applicable and the bounds in \Cref{lemma:stepsize} on the step sizes $\{\lambda_k\}$ are no longer required. If $\alpha < 1$ and combining \Cref{lemma: LS_SOL,lemma: LS_NC,lemma: LS_GD}~(iii), the key estimate \eqref{eq:KL_allnew} now takes the form
    \begin{equation*}\label{eq:KL_all}
    f(x_k) -f(x_{k+1}) \ge \sigma \|g_k\| \|s_k\| \min\{ \hat T, \hat T_{2,k},\hat T_{3,k},\hat T_{4,k} \},
    \end{equation*}
    where $\hat T := \min\{\bar C\cA^2,\cB, \frac{s\cZ}{2}, 1\}$, $\hat T_{2,k} := \bar C (a_k/s^{\alpha})^\frac{2}{1-\alpha}\|s_k\|^{\frac{2\alpha}{1-\alpha}}$, $\hat T_{3,k} := \frac{b_k}{s^\beta} \|s_k\|^{\beta}$, and $\hat T_{4,k} := \frac{s^{1-\zeta}}{2}z_k\|s_k\|^{\zeta}$. Here, $\bar C > 0$ is a (uniform) lower bound for $C_k$, i.e., we have $C_k = (\|\bar B_k\|+\|\bar B_k\|^2)^{-1} \geq \bar C$ for all $k \in \mathbb{N}$. (Such $\bar C$ exists due to \ref{B.1} and $\zeta_k \le \cZ$). The remaining steps are now essentially identical to the proof of \Cref{global_LS_KL_new}. We will omit further details here.
\end{proof}

Finally, we present a technical result on the rate of convergence of $\{\|g_k\|\}$ which is based on the {\L}ojasiewicz condition \ref{C.3}. \Cref{global_LS_KL_rate} will be used in the next section to establish saddle point avoidance properties of \Cref{alg:main_LS}.

\begin{lem}[Convergence rates]\label[lemma]{global_LS_KL_rate}
Let \ref{A.1}--\ref{A.2}, \ref{B.4}, \ref{B.3}, and \ref{C.1} be satisfied and let $\{x_k\}$ be generated by \Cref{alg:main_LS}. Suppose that $x^* \in \mathcal A$ is an accumulation point of $\{x_k\}$ at which property \ref{C.3} holds with exponent $\theta \in [0,1)$. Setting $\omega = \frac{(1+\max\{2\alpha, \zeta\})\theta}{1-\theta}$, we then have ${\liminf}_{k\to\infty}k^{1+p} \|g_k\|=0$ for all $p \in (0,\frac{1}{\omega-1})$.
\end{lem}

The proof of \Cref{global_LS_KL_rate} is relegated to \Cref{app:lem-rate}.

\subsection{Avoiding Strict Saddle Points}\label{ssec: LS3}

In this section, our goal is to prove that \Cref{alg:main_LS} can effectively avoid strict saddle points when we opt to use Hessian information, $B_k = \nabla^2 f(x_k)$. More specifically, we aim to show that the iterates generated by \Cref{alg:main_LS} can only converge to a second-order critical point $x^*$ satisfying $\nabla f(x^*) = 0$ and $\nabla^2 f(x^*) \succeq 0$.
We consider the following additional assumptions.

\begin{mdframed}[style=assumptionbox]
    \begin{enumerate}[label=\textup{\textrm{(D.\arabic*)}},topsep=2pt,itemsep=0ex,partopsep=0ex,leftmargin=6ex]
        \item \label{D.1} Defining $\mathcal N := \{k : \lambda_{\min}(\bar B_k) \leq 0\}$ and recalling $\mathcal N_{\mathrm{npc}} \equiv \{k \in \mathcal N: \mathrm{flag}_k = \NPC\}$, we assume $|\mathcal N| = \infty$ $\implies$ $|\mathcal N_{\mathrm{npc}} | = \infty$.
        \item \label{D.2} 
        We have $\zeta_k = \min\{\cZ,z_k\|g_k\|^\zeta\}$ where $\cZ > 0$, $\zeta \in (0,\frac13]$, $z_k \geq \bar z > 0$, and $z_k = \mathcal O((k\log(k)^2)^{\zeta})$, $k\to\infty$, and ${\sum}_{k=1}^\infty a_k^{-{1}/{\alpha}} < \infty$.
      \end{enumerate} 
\end{mdframed}

Under \ref{D.2}, the assumptions \ref{B.3} and \ref{C.1}, used in the previous sections, are naturally satisfied. We now present our main avoidance result. 

\begin{thm}[Saddle point avoidance] \label[theorem]{Avoid_LS}
Let \ref{A.1}--\ref{A.2}, \ref{B.4}, and \ref{D.1}--\ref{D.2} hold and let $\{x_k\}$ be generated by \Cref{alg:main_LS}. Suppose that the algorithm does not stop after finitely many iterations. Let $x^* \in \mathcal A$ be an accumulation point of $\{x_k\}$ at which property \ref{C.3} is satisfied. Then, $\{x_k\}$ converges to $x^*$ and we have $\nabla f(x^*) = 0$ and $\nabla^2 f(x^*) \succeq 0$.
\end{thm}

\begin{proof}
By \Cref{global_LS_KL_new}, we have $x_k\to x^*$ and $\|g_k\| \to 0$, and without loss of generality, we assume $\{x_k\} \subset B_{2r}(x^*) \subset V_H$, $r > 0$. Let $L_H$ further denote the Lipschitz constant of $\nabla^2 f$ on $B_{2r}(x^*)$. By Lemma~\ref{global_LS_KL_rate} and \ref{D.2}, we have
$$ 0\le\liminf_{k\to\infty} \zeta_k \le \liminf_{k \to \infty} z_k \|g_k\|^\zeta \le \limsup_{k\to\infty}\frac{M\log(k)^{2\zeta}}{k^{p\zeta}}\liminf_{k \to \infty} (k^{1+p} \|g_k\|)^\zeta = 0,$$
for some $M > 0$ and $p \in (0,\frac{1}{\omega-1})$. We now assume that $\nabla^2 f(x^*)$ is indefinite. This implies $\liminf_{k \to \infty}\lambda_{\min}(\bar B_k) = \liminf_{k \to \infty} \lambda_{\min}(\nabla^2 f(x_k)+ \zeta_k I) <0$ and by \ref{D.1}, we can infer $|\mathcal N| = \infty$ and $|\mathcal N_{\mathrm{npc}} | = \infty$. Without loss of generality, we assume $\lim_{\mathcal N_{\mathrm{npc}} \ni k \to \infty} \zeta_k = 0$. Using \ref{D.2} and \Cref{lemma: LS_NC} (ii), this yields
\[ \frac{3\rho(\sigma-1)d_k^\top B_kd_k}{L_H \|d_k\|^3} \geq  \frac{3\rho(1-\sigma) \zeta_k }{ L_H \|d_k\|} = \frac{3\rho(1-\sigma)}{L_H}\frac{z_k}{\|g_k\|^{1-\zeta}} \to \infty \]
as $\mathcal N_{\mathrm{npc}} \ni k \to \infty$. Hence, according to \Cref{lemma:stepsize}~(ii), there is $\bar k \in \mathbb N$ such that $\lambda_k \ge \frac{ 3\rho(\sigma-1) d_k^{\top} B_k d_k }{ L_H \|d_k\|^3} >  s$ for all $\mathcal N_{\mathrm{npc}} \ni k \geq \bar k$. Setting $\mathcal K := \{k \in \mathcal N_{\mathrm{npc}}: k \geq \bar k\}$ and using \ref{D.2}, \eqref{armijo2}, and \eqref{eq:sum1-new} (after potentially increasing $\bar k$), this implies
\begin{align*}
\infty > \varrho_{\bar k} & > \sum_{k \in \mathcal K} \frac{f(x_k) -f(x_{k+1})}{\sigma\|g_k\|} \ge \sum_{k \in \mathcal K} - \frac{\lambda_k  d_k^{\top} g_k + \frac{1}{2} \lambda_k^2  d_k^{\top} B_k d_k}{\|g_k\|}  \\
& \ge  \sum_{k \in \mathcal K} - \frac{\lambda_k^2  d_k^{\top} B_k d_k}{2\|g_k\|}  \ge \sum_{k \in \mathcal K} \frac{9\rho^2(\sigma-1)^2  (-d_k^{\top} B_k d_k)^3 }{2 L_H^2 \|d_k\|^6 \|g_k\|} \\
&\ge  \sum_{k \in \mathcal K}  \frac{4\rho^2(\sigma-1)^2}{ L_H^2} \frac{\zeta_k^3}{\|g_k\|} = \frac{4\rho^2(\sigma-1)^2}{L_H^2} \sum_{k \in \mathcal K} \frac{z_k^3}{\|g_k\|^{1-3\zeta}} =\infty,
\end{align*}
where $\varrho_{\bar k} = c(f(x_{\bar k})-f(x^*))^{1-\theta}$, cf. \ref{C.3}. This is a contradiction.
\end{proof}

\begin{rem}
Our analysis demonstrates that strict saddle point avoidance does not require the direct use of $\NPC$ directions associated with $\lambda_{min}(B_k)$; instead, a pseudo-negative eigenvalue $-\zeta_k$ suffices. This eigenpair-free mechanism reduces computational costs and can enhance practical performance. The key insight in our derivation is that if $x_k$ approaches a strict saddle point $x^*$, the condition \eqref{armijo2} and the forward linesearch forces the step size to increase in the order of $\sim1/\|g_k\|^{1-\zeta}$ as $\|g_k\| \to 0$. Coupled with \eqref{armijo2}, this implies that the terms $(f(x^k)-f(x^{k+1}))/\|g_k\|$, $k\in\mathcal N_{\mathrm{npc}}$, diverge and are not summable. However, this contradicts the KL-based guarantees shown in \Cref{global_LS_KL_new}. Our arguments are motivated by the analyses in \cite{gould2000exploiting,curtis2019exploiting,liu2022newton}, but we take additional advantage of the geometric properties provided by \ref{C.2}--\ref{C.3}. 
\end{rem}

\begin{rem}\label{rem:theta=0}
Condition \ref{D.1} means that if the matrix $\bar B_k$ contains $\NPC$ information for infinitely many iterations, then $\MR$ can detect such $\NPC$ information---at least on an infinite subsequence. We note that by setting $b_k=(k \log(k)^2)^{\beta}$ in $\theta_k$ and mimicking the derivations for $\zeta_k$ in the proof of \Cref{Avoid_LS}, it follows $\liminf_{k \to \infty} \theta_k =0$ and hence, the $\MR$-subproblems are being solved with increasing accuracy. In this scenario, \ref{D.1} can be ensured under reasonable assumptions. For instance, if the orthogonal projection of $g_k$ onto the eigenspace spanned by the non-positive eigenvectors of $\bar B_k$ is non-trivial, then we must have $\lambda_{\min}(T_k) = \lambda_{\min}(\bar B_k) \leq 0 $ at the last inner iteration of $\MR$ by \cite[eqn. 4.9]{paige1971computation}, which guarantees $\NPC$ detection. 
Moreover, \ref{D.1} holds with probability $1$ when setting $\theta_k =0$, $k \in \mathbb N$. In this case, if $\bar B_{k}$ is not positive definite, the dimension of the space spanned by the associated eigenvectors w.r.t. the positive eigenvalues is strictly less than $n$. Hence, due to $\theta_k = 0$ and with probability $1$, $\MR$ will return an $\NPC$ direction and will not stop early. Similar probabilistic arguments are also used in, e.g., \cite[Section 4]{gould2000exploiting}. Note that if the linear system is inconsistent, $\MR$ will output the nonzero residual vector $r_{g-1}$ as a zero curvature direction, cf. \cite[Lemma 3.4 (iii)]{liu2022minres}, where
$g$ is the grade of the system.
\end{rem}

\subsection{Fast Local Convergence to Non-Isolated Local Minima} \label{ssec: LS4}
As illustrated in the previous subsections, the analysis techniques based on the Kurdyka-{\L}ojasiewicz inequality allow us to ensure convergence of the iterates $x_k \to x^*$. We will now discuss finer local convergence properties under the assumption that the (stronger) {\L}ojasiewicz inequality \ref{C.3} holds at $x^*$, i.e., there is a neighborhood $V$ of $x^*$, $\theta \in (0,1)$, and $\eta, c > 0$ such that
\begin{equation} \label{eq:add:loja} [ f(x) - f(x^*) ]^{2\theta} \leq (2\mu)^{-1} \|\nabla f(x)\|^2, \end{equation}
for all $x \in V\cap \{x: 0 < f(x) - f(x^*) < \eta\}$, where $\mu:= 1/(2c^2(1-\theta)^2)$. Moreover, we are interested in scenarios where $x^*$ is a (potentially non-isolated) local minimum of the objective function $f$. To this end, let us define 
\[ \mathcal S := \{x \in \bR^n: \; \text{$x$ is a local minimum and} \; f(x) = f(x^*)\}. \] 
The {\L}ojasiewicz inequality \eqref{eq:add:loja} is known to imply the quadratic growth condition
\begin{equation} \label{eq:add:qg} f(x) - f(x^*) \geq [(1-\theta)\sqrt{2\mu}]^{\frac{1}{1-\theta}} \dist(x,\mathcal S)^{\frac{1}{1-\theta}} \quad \forall~x \in V^\prime, \end{equation}
where $V^\prime$ is a suitable neighborhood of $x^* \in \mathcal S$, cf.
\cite[Proposition 2.2]{rebjock2024fast1}. Combining \eqref{eq:add:loja} and \eqref{eq:add:qg}, the following error bound condition must hold in a neighborhood of the local minimum $x^* \in \mathcal S$:
\begin{equation}\label{eq:eb}
    (1-\theta)(2\mu)^{\frac{1}{2\theta}} \dist(x,\mathcal S) \leq \|\nabla f(x)\|^{\frac{1-\theta}{\theta}}.
\end{equation}
We refer to \cite{drusvyatskiy2021nonsmooth,bai2022errorbound,rebjock2024fast1} for more discussion on the conditions \eqref{eq:add:loja}, \eqref{eq:add:qg}, and \eqref{eq:eb}.
Based on these specialized conditions, we now show that fast local convergence can be achieved with rates that match the results in \cite{rebjock2024fast2}. Our proof strategy is based on the proper choice of the regularization parameters $\{\zeta_k\}$. In particular, the tolerances $\{\theta_k\}$ and the inherent properties of $\MR$ play a less crucial role---which is in contrast to the overall derivations in \cite{rebjock2024fast2}. 
In the following lemma, we first prove that, in a local region of $x^*$, the linear system \eqref{eq:subreg} is consistent and $\MR$ will always return $\mathrm{SOL}$-directions. This also allows us to strengthen some of the previous results in \Cref{lemma: LS_SOL}.

\begin{lem}\label{lemma:local1}
Let \ref{A.1}--\ref{A.2}, \ref{B.4}, \ref{B.3}, and \ref{C.1} be satisfied and let $\{x_k\}$ be generated by \Cref{alg:main_LS}. Suppose that $x^* \in \mathcal A \cap \mathcal S$ is an accumulation point of $\{x_k\}$ at which property \ref{C.3} holds with exponent $\theta \in (0,1)$. We further assume that $\zeta \leq \frac{1-\theta}{\theta}$, $\zeta_k \ge 2\min\{\cA,a_k\|g_k\|^\alpha\}$ for all $k$, and $\theta_k = \mathcal O(\zeta_k)$, $k \to \infty$. There then exists $\bar k \in \mathbb N$ such that:
\begin{enumerate}[label=\textup{(\roman*)},topsep=1ex,itemsep=0ex,partopsep=0ex, leftmargin = 25pt]
\item All directions $d_k$, $k \geq \bar k$, returned by $\MR$ are $\mathrm{SOL}$-directions. 
\item It holds that $-d_k^{\top}g_k \geq \|g_k\|^2 /M_1$, $\|d_k\| \le M_2 \max\{\|g_k\|, \|g_k\|^{\frac{1-\theta}{\theta}}\}$, and $\bar B_k \succeq \frac12\zeta_kI$ for all $k\geq \bar k$ and some constants $M_1, M_2>0$.
\end{enumerate}  
\end{lem}

\begin{proof}
We note that there is a neighborhood of $x^*$ such that the Euclidean projection $\mathcal P_{\mathcal S}$ onto $\mathcal S$ is well-defined (i.e., the set $\mathcal P_{\mathcal S}(x)$ is non-empty for all $x$ sufficiently close to $x^*$, \cite[Lemma 1.4]{rebjock2024fast1}). Moreover, due to \ref{C.3} and $x^* \in \mathcal S$, the error bound condition \eqref{eq:eb} needs to hold in a neighborhood of $x^*$. Thus, by the local Lipschitz continuity of the eigenvalues of a matrix, we can infer
\begin{equation} \label{eq:add:ew} \lambda_i(\nabla^2 f(x)) - \lambda_i(\nabla^2 f(\mathcal P_{\mathcal S}(x))) = \mathcal O(\dist(x,\mathcal S)) = \mathcal O(\|\nabla f(x)\|^{\frac{1-\theta}{\theta}}) \end{equation}
for all $i=1,\dots,n$ and $x$ in a neighborhood of $x^*$. In particular, by the definition of $\mathcal S$, there is $M > 0$ such that $\min_{i=1,\dots,n}\lambda_i(\nabla^2 f(x)) \geq -M \|\nabla f(x)\|^{\frac{1-\theta}{\theta}}$ for all $x$ close to $x^*$.  Due to $z_k \to \infty$, (this follows from \ref{C.1}), $\zeta \le \frac{1-\theta}{\theta}$, \eqref{eq:add:ew}, and $x_k \to x^*$ (by \Cref{global_LS_KL_new}), there are $\bar k$ and $M_1 >0$ such that 
\begin{equation*} \label{eq:add:low} B_k + \frac12 \zeta_k I \succeq 0, \quad \|\bar B_k\| \le \|B_k  - \nabla^2 f(x^*)\| + \|\nabla^2 f(x^*)\| + \zeta_k \le M_1\end{equation*}
for all $k\ge \bar k$. Hence, $\bar B_k$ is positive definite and the algorithm eventually will no longer return $\NPC$ directions. Furthermore, we have $d_k^\top \bar B_k d_k \geq \frac12 \zeta_k \|d_k\|^2$ and the curvature test will be passed due to  $\zeta_k \ge 2\min\{\cA,a_k\|g_k\|^\alpha\}$. Thus, no GD directions will be returned. Since $\bar B_k$ is now ensured to be positive definite, we can use \eqref{eq:pkgk} to obtain
\[ - d_k^{\top} g_k \ge  \frac{g_k^{\top} \bar B_k g_k}{g_k^{\top}\bar B_k \bar B_k g_k} \|g_k\|^2 \ge \frac{g_k^{\top} \bar B_k g_k}{\|\bar B_k\|\|\bar B_k^{1/2}g_k\|} \|g_k\|^2 \ge  \frac{1}{M_1} \|g_k\|^2. \]
Next, let us set $y_k:= \bar B_k^{-1} g_k = (B_k + \zeta_k I)^{-1} g_k$ and $\hat x_k := \mathcal P_{\mathcal S}(x_k)$. Then, by Taylor expansion and \ref{A.2}, it follows that
\begin{align*} y_k & = \bar B_k^{-1} [\nabla f(\hat x_k) + \nabla^2 f(\hat x_k)(x_k - \hat x_k) + \mathcal O(\|\hat x_k - x_k\|^2)] \\ & = \bar B_k^{-1} [ B_k (x_k - \hat x_k) + \mathcal O(\|\hat x_k - x_k\|^2)] \\ &= [ x_k-\hat x_k] + \bar B_k^{-1}[ \zeta_k (\hat x_k - x_k) + \mathcal O(\|\hat x_k - x_k\|^2)] \end{align*} 
and hence, using $\|\hat x_k - x_k\| = \dist(x_k,\mathcal S)$ and $\bar B_k \succeq \frac12 \zeta_k I$, this yields $\|y_k\| = \mathcal O(\dist(x_k,\mathcal S)) + \mathcal O(\zeta_k^{-1}\dist(x_k,\mathcal S)^2)$. Applying $\bar B_k d_k + r_k=-g_k$, the error bound \eqref{eq:eb}, $\theta_k = \mathcal O(\zeta_k)$, and $\zeta \le \frac{1-\theta}{\theta}$, there is some $M_2 >0$ such that
\begin{equation*}
    \begin{aligned}
        \|d_k\| & = \|\bar B_k^{-1}(r_k + g_k)\| \leq \|\bar B_k^{-1}\|\|r_k\| + \|y_k\| \\
        & \leq 2 \zeta_k^{-1}\theta_k \|g_k\| + \mathcal O(\max\{1,\zeta_k^{-1}\mathcal O(\dist(x_k,\mathcal S))\}\mathcal O(\dist(x_k,\mathcal S)))\\
        & \leq 2 \zeta_k^{-1}\theta_k \|g_k \| +  \mathcal O(\max\{\|g_k\|^{\frac{1-\theta}{\theta}},\zeta_k^{-1}\|g_k\|^{\frac{2-2\theta}{\theta}}\})\\
        & \leq M_2\max\{\|g_k\|, \|g_k\|^\frac{1-\theta}{\theta}\}
    \end{aligned}
\end{equation*}
(after increasing $\bar k$---if necessary). This finishes the proof.
\end{proof}


Now, we are ready to present the local superlinear convergence of $\{\|g_k\|\}$.

\begin{thm} \label[theorem]{thm:local-non}
Let \ref{A.1}--\ref{A.2}, \ref{B.4}, \ref{B.2}--\ref{B.3}, and \ref{C.1} be satisfied and let $\{x_k\}$ be generated by \Cref{alg:main_LS}. Suppose that the algorithm does not stop after finitely many iterations and $x^* \in \mathcal A \cap \mathcal S$ is an accumulation point of $\{x_k\}$ at which property \ref{C.3} holds with exponent $\theta = \frac12$. Furthermore, we assume $b_{k+1}/b_k \to 1$, $z_{k+1}/z_k \to 1$, 
 $\zeta_k \ge 2\min\{\cA,a_k\|g_k\|^\alpha\}$, $\theta_k = \mathcal O(\zeta_k)$, and $\sigma \in (0,\frac12)$. Then, the sequence $\{\|\nabla f(x_k)\|\}$ converges superlinearly to zero with order $1+\omega$ where $\omega \in (0,\min\{\beta,\zeta\})$.
\end{thm}

\begin{proof}
By \Cref{global_LS_KL_new}, we have $x_k\to x^*$ and $\|g_k\| \to 0$, and we may assume $\{x_k\} \subset B_{2r}(x^*) \subset V_H$, $r > 0$. Let $L_g, L_H$ again denote the Lipschitz constants of $\nabla f$ and $\nabla^2 f$ on $B_{2r}(x^*)$. 
By our assumptions on the parameters, Lemma~\ref{lemma:local1} is applicable (with $\theta = \frac12$), i.e., there is $\bar k \in \mathbb{N}$ such that $d_k = p_t$ for all $k\ge \bar k$. By \Cref{lemmaMR} (i), $\sigma \in (0,\frac12)$, and Lemma~\ref{lemma:local1} (ii), we have 
\begingroup
\allowdisplaybreaks
\begin{align} \nonumber
   &f(x_k+d_k)-f(x_k)-\sigma g_k^{\top} d_k  \le (1-\sigma) g_k^{\top} d_k+\frac{1}{2} d_k^\top B_k d_k+ \frac{L_H}{6}\|d_k\|^3 \\ \nonumber
   & \le (\sigma - 1) d_k^\top \bar{B}_k d_k+\frac{1}{2} d_k^\top B_k d_k+\frac{L_H}{6}\|d_k\|^3 \\ \nonumber
   & =\Big(\sigma-\frac{1}{2}\Big) d_k^{\top} \bar{B}_k d_k - \frac{1}{2} \zeta_k \|d_k\|^2+\frac{L_H}{6}\|d_k\|^3 \\ \nonumber
   & \leq \Big(\Big(\frac{\sigma}{2}-\frac34\Big)\zeta_k + \frac{L_H}{6}\|d_k\|\Big)\|d_k\|^2  \\
   & \le \Big(\Big(\frac{\sigma}{2}- \frac{3}{4}\Big) \min\{\cZ,z_k\|g_k\|^\zeta\} + \frac{L_HM_2}{6} \|g_k\|\Big) \|d_k\|^2 \label{eq:lambdaequal1}
   \end{align}
   \endgroup
   for all $k\geq \bar k$. Moreover, due to $\|g_k\| \to 0$ as $k \to \infty$ and $\zeta \le 1$, there exists $K \ge \bar k$ such that for all $k \ge K$ the right hand side of \eqref{eq:lambdaequal1} is non-positive. Thus, the step size $\lambda_k=1$ will be accepted in the Armijo linesearch for all $k \geq K$. Consequently, using Lemma~\ref{lemma:local1} (ii) and \eqref{armijo}, we obtain
    \begin{equation}
        \label{eq:KL_SOL3}
        \begin{aligned}
            \frac{f(x_k)-f(x_{k+1})}{\|g_k\|} &\ge \sigma \lambda_k \frac{-d_k^{\top}g_k}{\|g_k\|} \ge \frac{\sigma}{M_1} {\|g_k\|} \quad \forall~k \geq K. 
        \end{aligned}
        \end{equation}
    Following our earlier strategies (cf. \eqref{eq:use-this-one-new} and \eqref{eq:sum1-new} in the proof of \Cref{global_LS_KL_new}), summing \eqref{eq:KL_SOL3} for $k \geq j \geq K$ and applying \eqref{eq:add:loja} with $\theta = \frac12$, we have
    \[ \Gamma_j := {\sum}_{k=j}^\infty \|g_k\| \leq \frac{M_1c}{\sigma} \sqrt{f(x_j)-f(x^*)} \leq \frac{M_1c}{\sqrt{2\mu}\sigma} \|g_j\| = C_\theta (\Gamma_j - \Gamma_{j+1}), \]
    %
    %
    where $C_\theta = M_1c/(\sqrt{2\mu}\sigma)$.
    If $C_\theta \le 1$, we obtain $\Gamma_{j+1} \le 0$ and finite step convergence. Otherwise, we can infer $\Gamma_{j+1} \le (1-1/C_\theta) \Gamma_j$ which establishes q-linear convergence of $\{\Gamma_j\}$ and r-linear convergence of $\{\|g_j\|\}$. To continue we invoke the following technical auxiliary result (see \Cref{app:auxiliary} for a proof).

    \begin{lem} \label[lemma]{lem:auxiliary} Let $\{u_k\} \subset \bR_+$ be a sequence that converges r-linearly to $0$ and let $\{v_k\} \subset \bR_+$ be given with $v_{k+1}/v_k \to 1$. Then, we have $u_kv_k \to 0$. 
    \end{lem}

Hence, applying \Cref{lem:auxiliary}, we can infer $ b_k^{\gamma_1}\|g_k\|^{\gamma_2} \to 0$ and $z_k^{\gamma_1}\|g_k\|^{\gamma_2} \to 0$, $k\to\infty$, for any fixed $\gamma_1, \gamma_2 >0$. Thus, by \ref{A.2} and Lemma~\ref{lemma:local1}, and using $\|\bar B_k d_k + g_k\| = \|r_k\| \leq \theta_k\|g_k\|$, we have
    \begin{equation*}
        \begin{aligned}
            \|g_{k+1}\| & = \|g_{k+1} - g_{k} - B_k d_k + B_k d_k + g_{k}\|\\
            & \le  \|\nabla f(x_{k+1}) - \nabla f(x_{k}) - \nabla^2 f(x_k) d_k\| + \zeta_k \|d_k\| + \|r_k\| \\
            & \le \frac{L_H}{2}\|d_k\|^2 + \min\{\cZ, z_k\|g_k\|^{\zeta}\} M_2 \|g_k\| +  \min\{\cB, b_k\|g_k\|^{\beta}\} \|g_k\|\\
            & \le \frac{L_H M_2^2}{2}\|g_k\|^2 + M_2z_k\|g_k\|^{1+\zeta} +  b_k\|g_k\|^{1+\beta} \\
        \end{aligned}
    \end{equation*}
    as $k\to\infty$. This finishes the proof.
\end{proof}

The different assumptions on the parameters in \Cref{thm:local-non} are satisfied for the exemplary and simple choices:
\[ \alpha=\beta=\zeta, \quad 2a_k = z_k = (k\log(k)^2)^\zeta, \quad b_k \leq z_k, \quad 2\cA = \cZ, \quad \cB \leq \cZ, \]
and $b_{k+1}/b_k \to 1$. In this case, setting $\alpha=\beta=\zeta = 1$, the rate of convergence can be arbitrarily close to quadratic convergence. 

\begin{rem}
The parameter $\zeta_k$ plays a crucial role in our derivations. In the global convergence analysis, it allows us to obtain descent estimates along $\NPC$ directions. Moreover, $\zeta_k$ serves as a pseudo-eigenvalue of the matrix $\bar B_k$ which facilitates the avoidance of strict saddle points in an eigenpair-free manner. By properly choosing $\zeta_k$ and the parameters in the curvature test \eqref{eq:cur}, we can further ensure that $\bar B_k$ is positive definite. In particular, the linear system \eqref{eq:subreg} is consistent around non-isolated minima that satisfy the {\L}ojasiewicz inequality with $\theta = \frac12$. This then brings many fruitful properties for the local analysis.
\end{rem}
\section{Numerical Experiments}\label{sec: Num}

In this section, we report the numerical performance of the proposed linesearch approach on different problems. We first list implementational details in \Cref{ssec: Num1}. We then numerically confirm and illustrate the local superlinear convergence properties of Algorithm~\ref{alg: MINRES_LS} (as discussed in \Cref{thm:local-non}) in \Cref{ssec: Num2}. This test is performed using Matlab on a MacBook Pro with 2 GHz Quad-Core Intel Core i5 and 16 GB memory. Additional tests on a large-scale auto-encoder problem and a series of problems from the CUTEst test collection are conducted in Sections~\ref{ssec: Num3} and \ref{ssec: Num4}. These experiments are run using Python on a server with NVIDIA GeForce RTX 3090 and 24 GB memory size.
\subsection{Implementation details}\label{ssec: Num1}
Our numerical comparisons will be based on the following algorithms: 
	\begin{itemize}
		\item \textbf{Newton-CR} \cite{dahito2019conjugate}, \textbf{Newton-CG} \cite{dembo1983truncated}: The classical Newton method with CR or
        CG as linear system solver; the step-size is generated using the strong Wolfe-Powell linesearch condition \cite{more1994line}.
		\item \textbf{Armijo-L-BFGS} \cite{li2001global,kanzow2023regularization}: We test the L-BFGS method with Armijo linesearch and a cautious updating scheme; in particular, the curvature pair $\{y_k, s_k\}$ is only used when $y_k^{\top} s_k \ge 10^{-18} \|s_k\|^2$.
		\item \textbf{Wolfe-L-BFGS} \cite{liu1989limited}: We implement the classical Liu-Nocedal L-BFGS method with a strong Wolfe-Powell linesearch condition.\footnote{Note that the strong Wolfe-Powell linesearch condition in \textbf{Newton-CR}, \textbf{Newton-CG} and \textbf{Wolfe-L-BFGS} is based on the Mor\'e-Thuente line search  \cite{more1994line}, which differs from the linesearch with zoom \cite[Algorithm 3.5]{jorge2006numerical} used in \cite[Section 4]{liu2022newton}.}
\end{itemize}  
We implement \Cref{alg:main_LS} with L-BFGS approximations and full Hessians, which are referred to as \textbf{L-BFGS-MR} and \textbf{Newton-MR}.  \textbf{L-BFGS-MR} uses a slightly adjusted curvature test (see \eqref{eq:cur2}) and \textbf{Newton-MR} uses the original test \eqref{eq:cur}. The L-BFGS matrices are built via the compact form \cite{byrd1994representations}:
\begin{align*}
    S_k&=\left[ s_{j_0}, \dots, s_{j_{\tilde{m}-1}} \right], \quad Y_k=\left[ y_{j_0}, \dots, y_{j_{\tilde{m}-1}} \right], \quad \tilde{m}=\min\{m,k+1\}\\
    B_{k}^{\mathrm{BFGS}}&=\gamma_k I-\left[\begin{array}{ll}
            \gamma_k S_{k} & Y_{k}
        \end{array}\right]\left[\begin{array}{cc}
            \gamma_k S_{k}^{\mathrm{\top}} S_{k} &     L_{k} \\
        L_{k}^{\mathrm{\top}} & -D_{k}
        \end{array}\right]^{-1}\left[\begin{array}{ll}
            \gamma_k S_{k} & Y_{k}
            \end{array}\right]^{\top}.
\end{align*}
Here, $L_k$ is the strictly lower part of $S_k^{\top} Y_k$, $D_k$ is the diagonal part of $S_k^{\top} Y_k$,  $R_k$ is the non-strictly upper part of $S_k^{\top} Y_k$. We choose $m=10$ and $ \gamma_k=y_k^{\top} y_k/ y_k^{\top} s_k$. In contrast to \textbf{Armijo-L-BFGS}, we use the pair $\{y_k, s_k\}$ to update $B_k^{\mathrm{BFGS}}$ when $|y_k^{\top} s_k| \ge {10^{-18}} \|s_k\|^2$. Thus, the L-BFGS matrix can be generally indefinite which allows us to explore possible non-positive curvature.

All algorithms use the same initial points generated from the uniform distribution on $[0,1]^{n}$. We set the maximum number of iterations for all subproblem solvers to be 1000. We terminate algorithms when $\|g_k\| \le 10^{-10}$, the step-size is less than $10^{-18}$, or the number of oracle calls reaches $10^{5}$. We set $s=1$, $\rho=0.5$, $\sigma=10^{-4}$ in \Cref{alg: LS} and $\sigma_{\text{Wolfe}}=0.9$ in the Wolfe-Powell linesearch. In \eqref{eq:cur}, \eqref{eq:cur2} and \ref{B.3}, we set $\cA=10^{-12}/2$, $\cZ=10^{-12}$, $a_k=(k\log(k)^2)^{\alpha}/2$ and $z_k=(k\log(k)^2)^{\zeta}$ with $\alpha=\zeta=1$ and $\bar M=10^8$. We choose $\theta_k=\min\{0.1, \sqrt{\|g_k\|}\}$ for all truncated Newton methods and  $\theta_k=\min\{0.1, \log(k) \sqrt{k\|g_k\|}\}$ for \textbf{L-BFGS-MR}.

\begin{rem}
If we replace \eqref{eq:cur} with the following refined curvature test
	   \begin{equation}\label{eq:cur2}
		 \begin{cases}
		   p_t^{\top}\bar{B}_k p_t \ge \min\{\cA, a_k \|g_k\|^\alpha\}\max\{\|p_t\|^2, \|g_k\|^2\}, & \text{if $d_k=p_t$},\\
		   |d_k^{\top}B_kd_k| <  \bar M \|d_k\|^2, & \text{if $d_k=r_{t-1}$},
		 \end{cases}
	   \end{equation}
	   $\bar M>0$, then Theorems~\ref{global_LS} and \ref{global_LS_KL} hold without \ref{B.1}. Specifically, Lemma~\ref{lemma: LS_SOL} (i) is satisfied with $C_k=1$ by Theorem~\ref{lemmaMR}~(i) and the first inequality in \eqref{eq:cur2}. Moreover, using the second inequality in \eqref{eq:cur2}, we can rewrite \eqref{eq:above1} as
	   $$(\sigma-1) \theta_{k_\ell} \|g_{k_\ell}\|^2 < \rho^{-1} \lambda_{k_{\ell}} d_{k_\ell}^\top (\hat B_{k_\ell}-B_{k_\ell}) d_{k_\ell}/2 \le \rho^{-1} \lambda_{k_{\ell}}(\|\hat B_{k_\ell}\| +\bar M)\|d_{k_\ell}\|^2/2, $$
	   which also leads to a contradiction as $\ell \to \infty$. Thus, Theorems~\ref{global_LS} and \ref{global_LS_KL} remain valid. The test \eqref{eq:cur2} can be used when the approximations $\{B_k\}$ are chosen as indefinite and potentially unbounded (quasi-Newton) matrices.
\end{rem}

\begin{figure}[t]
\begin{tikzpicture}[scale=1]
	\node[right] at (0,0) {\includegraphics[width=3.5cm,trim=46 0 0 0,clip]{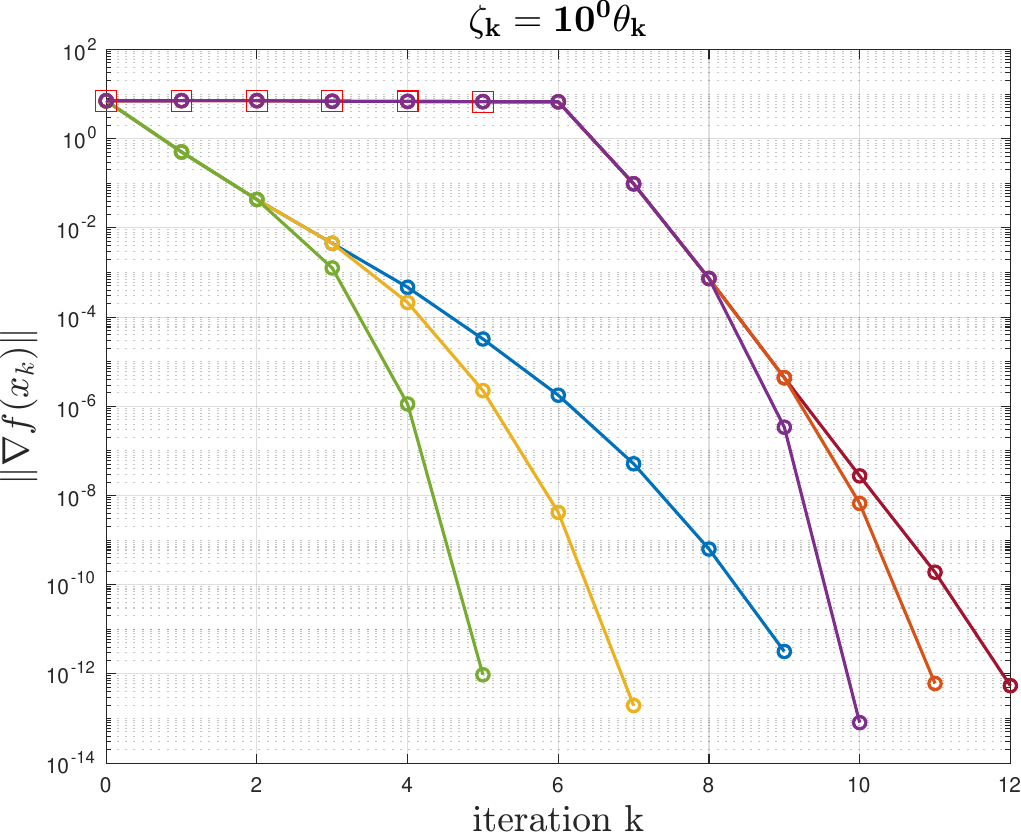}};
	\node[right] at (3.8,0) {\includegraphics[width=3.5cm,trim=46 0 0 0,clip]{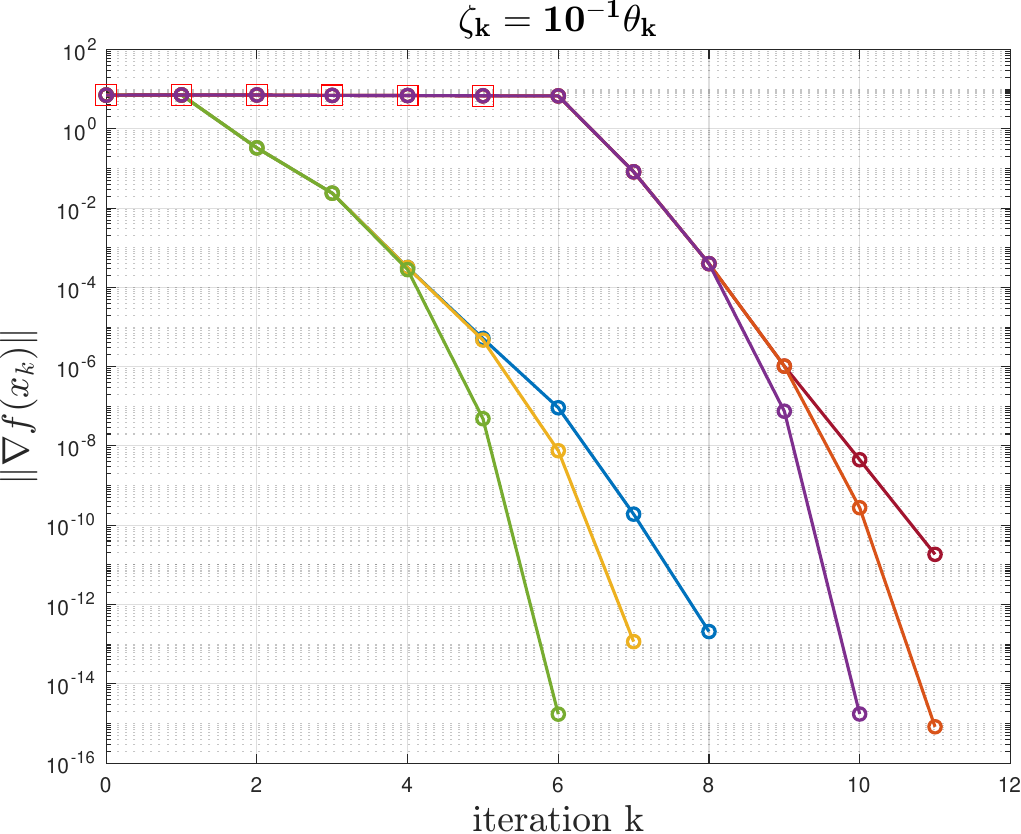}};
	\node[right] at (7.6,0) {\includegraphics[width=3.5cm,trim=46 0 0 0,clip]{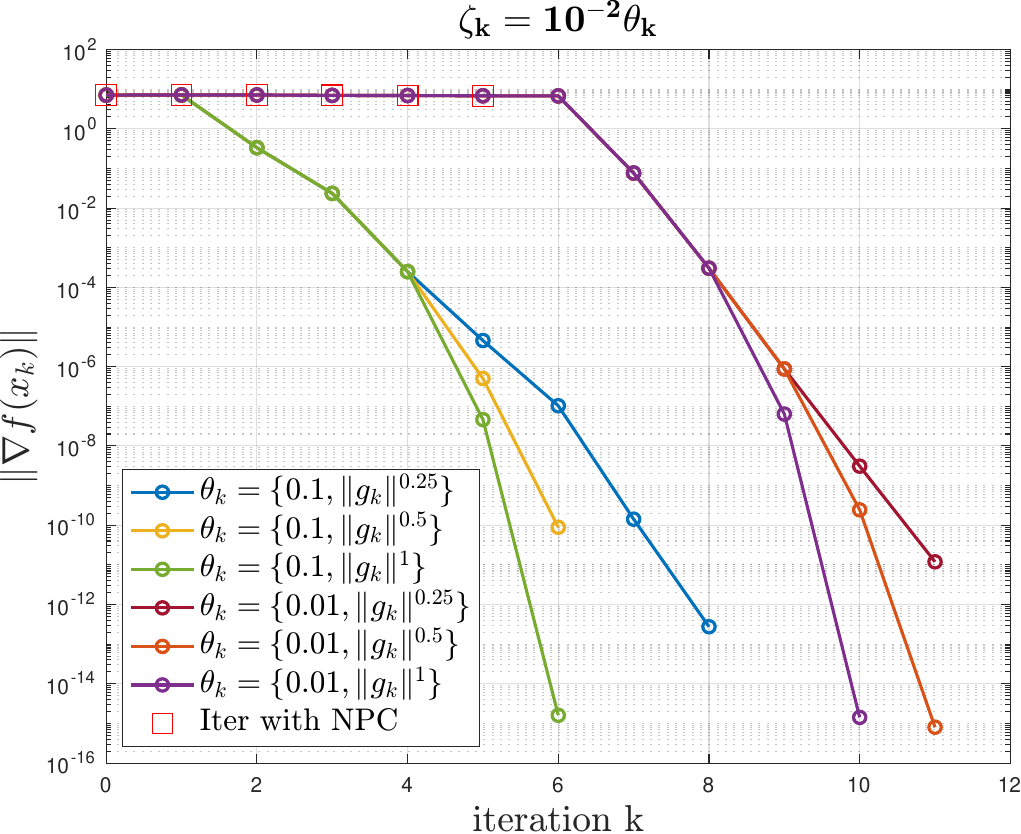}};
    \draw (11.6,1.39)--(11.6,-1.31);
    \draw (11.55,-1.31)--(11.6,-1.31);
    \draw (11.55,1.39)--(11.6,1.39);
    \draw [fill=white,white] (11.55,0.6) rectangle (11.65,-0.52); 
    \node at (11.6,0.04) {\rotatebox{-90}{{\tiny $\|\nabla f(x_k)\|$}}};
\end{tikzpicture}

	\caption{Convergence of $\{\|g_k\|\}$ on the toy function \eqref{eq:prob-toy} for different choices of $\theta_k$ and $\zeta_k$. Iterations where NPC information is detected are marked with a red square.}
	\label{fig:local}
\end{figure}
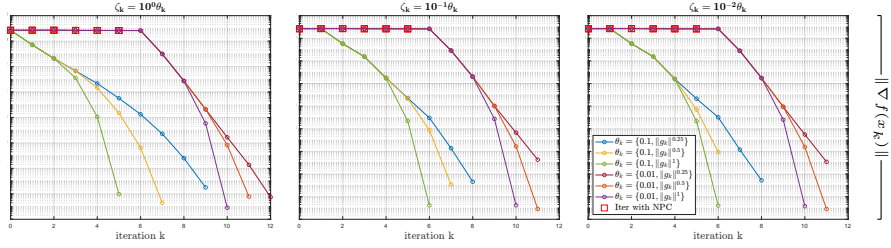

\subsection{Testing local superlinear convergence}\label{ssec: Num2}
We first use the test function $f$ in \cite[Appendix A]{rebjock2024fast2} to empirically verify local superlinear convergence of \textbf{Newton-MR} (as discussed in \Cref{thm:local-non}):
\begin{equation} \label{eq:prob-toy} f(x, y): =\frac{1}{2}\|y-\sin (x)\|^2, \quad x\in \bR^n, \quad y\in \bR^n.
\end{equation}
The points in the set $\{(x,y) \in \bR^n \times  \bR^n : y=\sin(x)\}$ are the global solutions of this problem. Furthermore, in \cite[Appendix A]{rebjock2024fast2}, it is shown that $f$ satisfies the PL condition and therefore \ref{C.3} is valid with $\theta=\frac12$. We choose $n=200$, the inexactness parameter $\theta_k=\min\{\mathsf{B}, \|g_k\|^{\beta}\}$ with $\mathsf{B} \in \{0.01,0.1\}$, $\beta \in\{1,\frac12,\frac14\}$, and $\zeta_k=\zeta \theta_k$ where $\zeta\in\{1,10^{-1},10^{-2}\}$. The results are depicted in \Cref{fig:local}; the superlinear and nearly quadratic convergence of $\{\|g_k\|\}$ is clearly visible for the different choices of $\beta$ and $\zeta$--aligning with the guarantees proven in \Cref{thm:local-non}. Furthermore, smaller choices of $\theta_k$ and $\zeta_k$ seem to lead to more iterations that detect NPC information. 

\begin{table}[t]
\centering
\setlength{\tabcolsep}{5pt}
\NiceMatrixOptions{cell-space-limits=1pt}
\begin{NiceTabular}{|c| c c c c|}%
\toprule
\Block{1-1}{Dataset} & \Block{1-1}{${n}$} & \Block{1-1}{${d}$} & \Block{1-1}{Encoder network architecture} & \Block{1-1}{Reference} \\ \Hline
\Block{1-1}{\texttt{CIFAR10}} & \Block{1-1}{$50,000$} & \Block{1-1}{$1,664,232$} & \Block{1-1}{$3,072-256-128-64-32-16-8$} & \cite{krizhevsky2009learning} \\ 
\Block{1-1}{\texttt{STL10}} & \Block{1-1}{$5,000$} & \Block{1-1}{$8,365,504$} & \Block{1-1}{$27,648-150-100-50-25-6$} & \cite{coates2011analysis} \\
\bottomrule 
\end{NiceTabular}
\caption{Auto-encoder model for \texttt{CIFAR10} and \texttt{STL10} datasets.}
\label{table: auto}
\end{table}

\begin{table}[t]
\centering
\setlength{\tabcolsep}{5pt}
\NiceMatrixOptions{cell-space-limits=1pt}

\begin{NiceTabular}{|c| p{0.8cm}p{1.6cm}p{1.6cm}p{1cm}p{1cm}p{0.8cm}|}%
 [ 
   code-before = 
    \rectanglecolor{lavender!30}{2-2}{7-2}
    \rectanglecolor{lavender!30}{2-4}{7-4}
    \rectanglecolor{lavender!30}{2-6}{7-6}
 ]
\toprule
\Block[c]{1-1}{Dataset: \texttt{CIFAR10}} & \Block{1-1}{Iter} & \Block{1-1}{Fun} & \Block{1-1}{$\|g_k\|$} & \Block{1-1}{Oracles} & \Block{1-1}{Time(s)} & \Block{1-1}{Ratio}  \\ \Hline
\Block{1-1}{\textbf{L-BFGS-MR}} & \Block{1-1}{303.7} & \Block{1-1}{$8.629\cdot10^{-2}$} & \Block{1-1}{$7.94\cdot10^{-11}$} & \Block{1-1}{687.8} & \Block{1-1}{163.2} & \Block{1-1}{100\%} \\
\Block{1-1}{\textbf{Armijo-L-BFGS}} & \Block{1-1}{323.7} & \Block{1-1}{$8.784\cdot10^{-2}$} & \Block{1-1}{$4.16\cdot10^{-10}$} & \Block{1-1}{789.7}  & \Block{1-1}{160.2}  & \Block{1-1}{0\%} \\
\Block{1-1}{\textbf{Wolfe-L-BFGS}} & \Block{1-1}{286.4} & \Block{1-1}{$8.759\cdot10^{-2}$} & \Block{1-1}{$4.06\cdot10^{-10}$} & \Block{1-1}{1250.0} & \Block{1-1}{498.7} & \Block{1-1}{0\%} \\
\Block{1-1}{\textbf{Newton-MR}} & \Block{1-1}{35.7} & \Block{1-1}{$8.287\cdot10^{-2}$} & \Block{1-1}{$9.97\cdot10^{-12}$} & \Block{1-1}{601.0} & \Block{1-1}{333.4} & \Block{1-1}{100\%} \\
\Block{1-1}{\textbf{Newton-CR}} & \Block{1-1}{49.8} & \Block{1-1}{$8.306\cdot10^{-2}$} & \Block{1-1}{$9.17\cdot10^{-12}$} & \Block{1-1}{799.3} & \Block{1-1}{424.5} & \Block{1-1}{100\%} \\
\Block{1-1}{\textbf{Newton-CG}} & \Block{1-1}{45.6} & \Block{1-1}{$8.632\cdot10^{-2}$} & \Block{1-1}{$9.71\cdot10^{-12}$} & \Block{1-1}{844.8} & \Block{1-1}{447.9} & \Block{1-1}{100\%} \\
\midrule
\end{NiceTabular}
\begin{NiceTabular}{|c| p{0.8cm}p{1.6cm}p{1.6cm}p{1cm}p{1cm}p{0.8cm}|}%
 [ 
   code-before = 
    \rectanglecolor{lavender!30}{2-2}{7-2}
    \rectanglecolor{lavender!30}{2-4}{7-4}
    \rectanglecolor{lavender!30}{2-6}{7-6}
 ]
\midrule
\Block{1-1}{Dataset: \texttt{STL10}} & \Block{1-1}{Iter} & \Block{1-1}{Fun} & \Block{1-1}{$\|g_k\|$} & \Block{1-1}{Oracles} & \Block{1-1}{Time(s)} & \Block{1-1}{Ratio} \\ \Hline
\Block{1-1}{\textbf{L-BFGS-MR}} & \Block{1-1}{150.1} & \Block{1-1}{$1.265\cdot10^{-1}$}   & \Block{1-1}{$7.81\cdot10^{-11}$} & \Block{1-1}{370.0} & \Block{1-1}{101.4} & \Block{1-1}{100\%} \\
\Block{1-1}{\textbf{Armijo-L-BFGS}} & \Block{1-1}{212.9} & \Block{1-1}{$1.345\cdot10^{-1}$} & \Block{1-1}{$3.14\cdot10^{-10}$} & \Block{1-1}{545.3} & \Block{1-1}{79.6} & \Block{1-1}{0\%} \\
\Block{1-1}{\textbf{Wolfe-L-BFGS}} & \Block{1-1}{156.4} & \Block{1-1}{$1.341\cdot10^{-1}$} & \Block{1-1}{$3.12\cdot10^{-10}$} & \Block{1-1}{949.6} & \Block{1-1}{239.6} & \Block{1-1}{0\%} \\
\Block{1-1}{\textbf{Newton-MR}} & \Block{1-1}{25.2} & \Block{1-1}{$1.247\cdot10^{-1}$} & \Block{1-1}{$1.38\cdot10^{-11}$} & \Block{1-1}{493.3} & \Block{1-1}{168.5} & \Block{1-1}{100\%} \\
\Block{1-1}{\textbf{Newton-CR}} & \Block{1-1}{42.3} & \Block{1-1}{$1.292\cdot10^{-1}$}    & \Block{1-1}{$1.10\cdot10^{-10}$} & \Block{1-1}{775.0} & \Block{1-1}{230.5} & \Block{1-1}{97\%} \\
\Block{1-1}{\textbf{Newton-CG}} & \Block{1-1}{33.1} & \Block{1-1}{$1.254\cdot10^{-1}$} & \Block{1-1}{$7.14\cdot10^{-12}$} & \Block{1-1}{675.1} & \Block{1-1}{211.7} & \Block{1-1}{100\%} \\
\bottomrule
\end{NiceTabular}

    \caption{Numerical results for the auto-encoder problem. Data are recorded at termination, triggered by either ``convergence'' ($\|g_k\| < 10^{-10}$) or ``step-size stagnation'' ($\alpha_k < 10^{-18}$). In particular, ``Iter'', ``Fun'', and ``Oracle'' are the total number of (outer) iterations, the function value, and the oracle calls, respectively. An oracle call refers to an operation equivalent to a single function evaluation in terms of complexity. Time is measured in seconds. The results are averaged over 100 independent runs. Lastly, ``Ratio'' represents the percentage that the algorithm terminates with convergence among the 100 runs.}
    \label{table: time}
\end{table}

\subsection{Auto-encoder problem}\label{ssec: Num3}
We consider the following nonconvex deep auto-encoder problem 
\begin{equation} \label{eq:prob-auto} f(\mathbf{x}_{\mathcal{E}}, \mathbf{x}_{\mathcal{D}})=\frac{1}{n} \sum_{i=1}^n\|\mathbf{a}_i-\mathcal{D}(\mathcal{E}(\mathbf{a}_i ; \mathbf{x}_{\mathcal{E}}) ; \mathbf{x}_{\mathcal{D}})\|^2+\lambda \psi(\mathbf{x}_{\mathcal{E}}, \mathbf{x}_{\mathcal{D}}), \end{equation}
where $\mathcal{E}: \bR^{p} \to \bR^{q}$ and $\mathcal{D}: \bR^{\ell} \to \bR^{p}$ are encoder and decoder mappings, respectively. Here, $x \in \bR^d$ is decomposed via $x =(\mathbf{x}_{\mathcal{E}}^{\top},\mathbf{x}_{\mathcal{D}}^{\top})^{\top}$; we refer to \cite{hinton2006reducing,martens2010deep,xu2020second} for more details. We use the same structure of the neural network and general setup as in \cite{liu2022newton}. To be more specific, the regularization function $\psi$ is chosen as $\psi(x)=\sum_{i=1}^{d} x_i^2/(1+x_i^2)$ and we set $\lambda=10^{-3}$. The initial point is drawn from a normal distribution with zero mean and standard derivation $10^{-8}$. We run experiments on the datasets \texttt{CIFAR10} and \texttt{STL10}\footnote{The \texttt{STL10} dataset contains colored images in ten classes. We relabel the even classes as ``0'' and the odd ones as ``1''.}; see \Cref{table: auto} for further information. 

\begin{figure}[t]
  \centering
  \subfigure{
  \includegraphics[width=5.5cm]{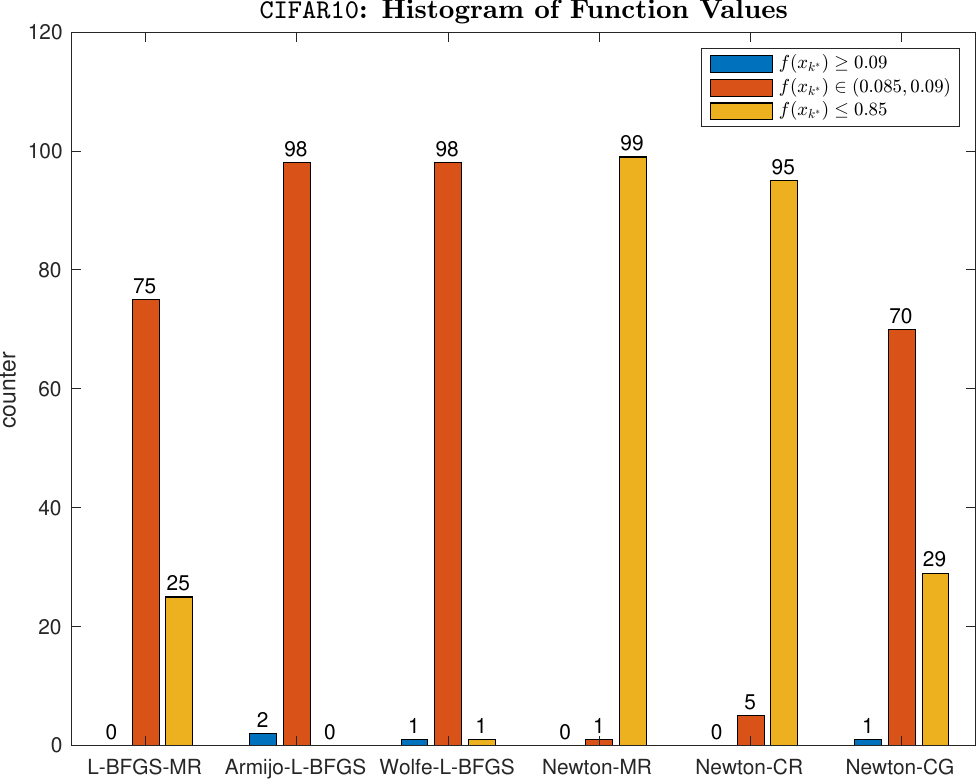}
  }
  \subfigure{
  \includegraphics[width=5.5cm]{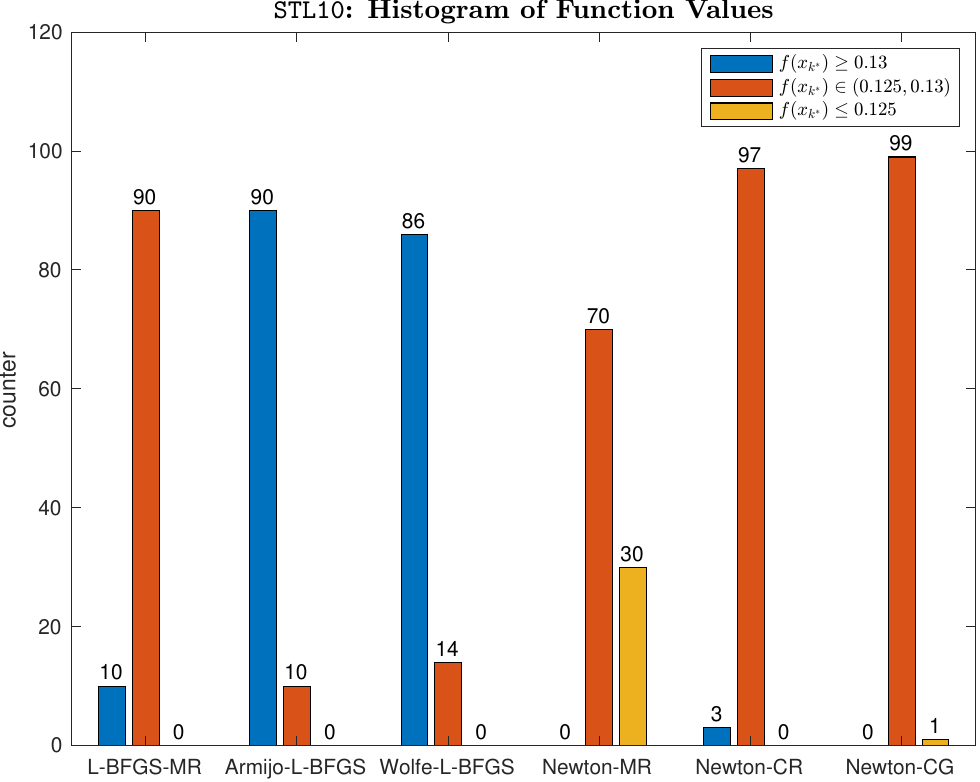}
  }
  \caption{Histograms of the final function values returned by the different algorithms when applied to the auto-encoder problem. The reported function values, $f(x_{k^*})$, are taken as the average over 100 independent runs.}
  \label{fig:count}
\end{figure}

\begin{figure}[t]
	\centering
	\includegraphics[scale=0.3]{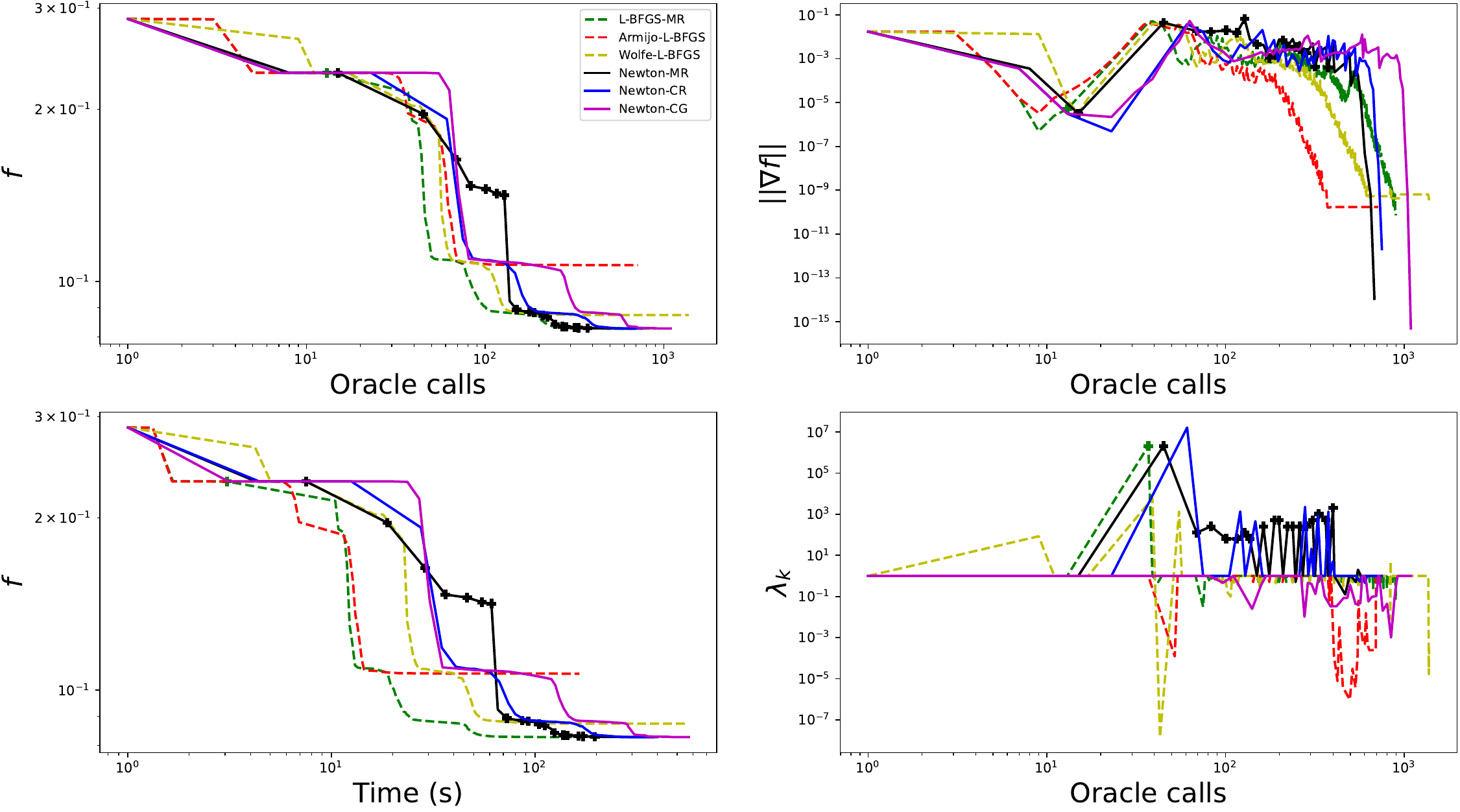}
    \caption{Performance of the different algorithms on the deep auto-encoder problem \eqref{eq:prob-auto} using the \texttt{CIFAR10} dataset. Changes of the function values, the norm of the gradient, and the step sizes w.r.t. to the number of oracle calls and cpu-time are reported. The markers indicate iterations with negative curvature detection.}
    \label{cifar}
\end{figure}

\begin{figure}[t]
	\centering
	\includegraphics[scale=0.3]{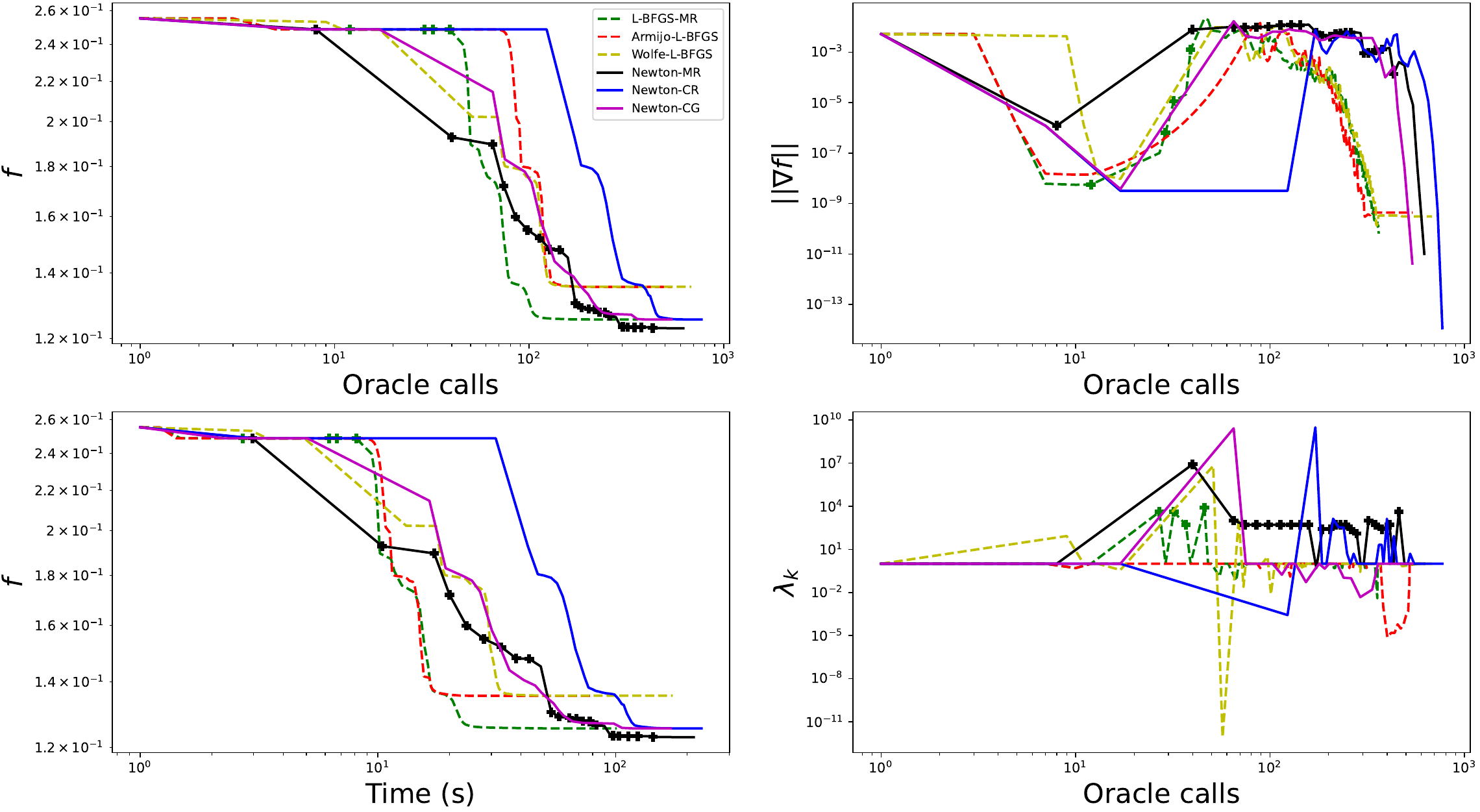}
    \caption{Performance of the different algorithms on the deep auto-encoder problem \eqref{eq:prob-auto} using the \texttt{STL10} dataset. Changes of the function values, the norm of the gradient, and the step sizes w.r.t. to the number of oracle calls and cpu-time are reported. The markers indicate iterations with negative curvature detection.} \label{STL10}
\end{figure} 

The average results for 100 independent runs are reported in \Cref{table: time}. Of the three truncated Newton methods, \textbf{Newton-MR} seems to perform the best w.r.t. the number of iterations, the objective function value, oracle calls, and computational time. Surprisingly, for quasi-Newton methods, only \textbf{L-BFGS-MR} manages to meet the convergence criteria and both \textbf{Armijo-L-BFGS} and \textbf{Wolfe-L-BFGS} stop prematurely due to small stepsizes. Compared to classical L-BFGS methods, \textbf{L-BFGS-MR}--equipped with negative curvature detection and truncation--can cope better with nonconvex landscapes and it achieves lower function values--albeit with slightly higher computational costs than \textbf{Armijo-L-BFGS}. This is also illustrated in \Cref{fig:count}, where we show a histogram/distribution of the optimal function values returned by the different methods. In Figures~\ref{cifar} and \ref{STL10}, we provide more detailed convergence plots comparing function and gradient values and stepsizes with respect to oracle calls and cpu-time. The plots complement our observations from \Cref{table: time} and \Cref{fig:count} and underline the benefits of exploiting NPC information in nonconvex problems--both for Newton and quasi-Newton methodologies.

\subsection{CUTEst test problems}\label{ssec: Num4}
Finally, we report the performance of the different algorithms on 237 test problems from the CUTEst environment. We use performance profiles for our comparison \cite{gould2016note,dolan2002benchmarking}, and follow the setting in \cite{liu2022newton}. We note that it is generally not known whether the test problems are convex or not. In this test, for \textbf{L-BFGS-MR}, we set $\theta_k=\min\{10^{-10}, \log(k)\sqrt{k\|g_k\|}\}$ to further enhance its performance and to exploit more NPC information at earlier iterations.

\begin{figure}[t]
  \centering
  \subfigure[Newton-type and L-BFGS-type ]{
  \includegraphics[width=5.5cm]{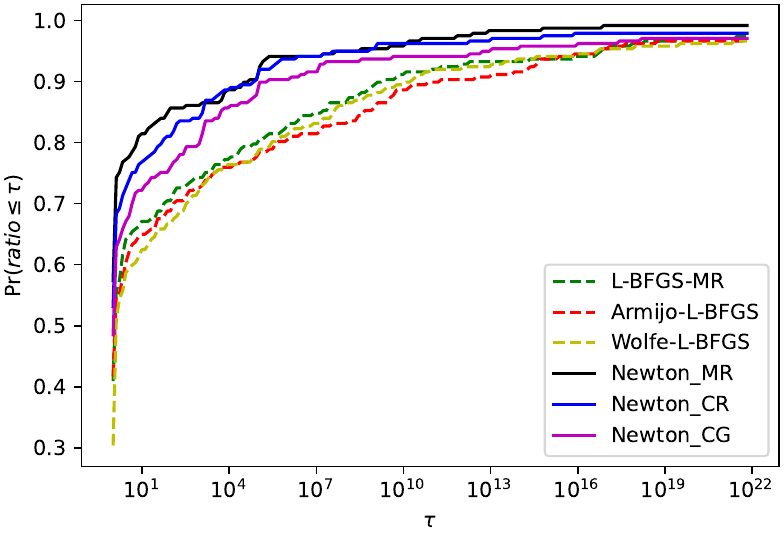}
  	\label{fig:cutest}
  }
  \subfigure[L-BFGS-type]{
  \includegraphics[width=5.5cm]{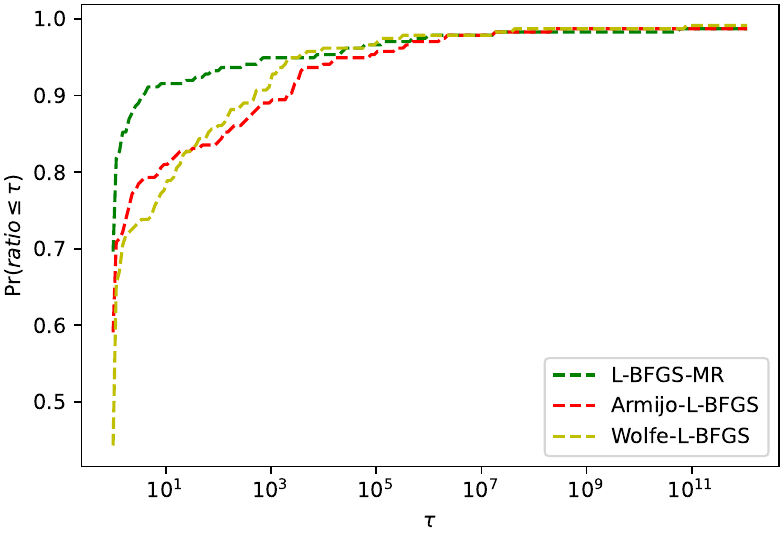}
  	\label{fig:cutestqn}
  }
  \caption{Performance profile in terms of $\{f(x_k)\}$. For a given $\tau$ in the $x$-axis, the corresponding value on the $y$-axis is the proportion of times that a given solver's performance lies within a factor $\tau$ of the best possible performance over all runs.}
  \label{fig: QN}
\end{figure}

We summarize the results in \Cref{fig: QN}. In \Cref{fig:cutest}, across all tested algorithms, \textbf{Newton-MR} is the best and slightly outperforms the other two truncated Newton methods. A more detailed comparison of the three L-BFGS-type algorithms is shown in \Cref{fig:cutestqn}. \textbf{L-BFGS-MR} clearly outperforms \textbf{Armijo-L-BFGS} and \textbf{Wolfe-L-BFGS} in terms of function values $\{f(x_k)\}$. These results further demonstrate the advantage of $\NPC$ detection and using $\MR$ as the inner linear system solver.

\section{Conclusion and Discussions}\label{sec: Con}
This work proposes a MINRES-based linesearch algorithm framework for solving unconstrained nonconvex optimization problems. The approach uses the well-known symmetric linear system solver $\MR$ to generate descent directions and adaptive mechanisms to control regularization, truncation, and acceptance parameters. We present comprehensive convergence guarantees under the Kurdyka-{\L}ojasiewicz property when general Hessian approximations $B_k \approx \nabla^2 f(x_k)$ are used. Moreover, when $B_k$ is the true Hessian, we establish stronger local properties including avoidance of strict saddle points and superlinear convergence to non-isolated minima. Finally, numerical experiments on an auto-encoder problem and the CUTEst problem collection underline the benefits of exploiting NPC information in nonconvex problems.

Due to the equivalence between CR and $\MR$ \cite{lim2024conjugate}, our eigenpair-free approach to avoid strict saddle points remains applicable when $\MR$ is replaced by CR. Extending our strategies to truncated Newton methods using CG is an interesting direction for future research.
Moreover, unlike \cite{rebjock2024fast2}, which requires a sophisticated analysis of CG in cases where the Hessian is merely positive semidefinite and the PL condition holds, this work employs adaptive regularization to ensure positive definiteness of the linear system, thus avoiding such discussions in the local analysis. It would also be valuable to analyze the performance of the $\MR$ steps--without regularization--as in \cite{rebjock2024fast2} and to establish comparable theoretical guarantees. 

\section*{Statements and Declarations} \vspace{-1ex}
\noindent\textbf{Funding and/or Conflicts of interests/Competing interests.}  Andre Milzarek was partly supported by the Shenzhen Science and Technology Program (No$.$ RCYX20221008093033010) and by Shenzhen Stability Science Program 2023, Shenzhen Key Lab of Multi-Modal Cognitive Computing.

\appendix
\section{Proof of \texorpdfstring{\Cref{global_LS_KL_rate}}{Lemma 3.10}} \label{app:lem-rate}

\begin{proof}
We first note that all steps in the proof of \Cref{global_LS_KL_new} are applicable---ensuring $x_k \to x^*$. The derivation of \Cref{global_LS_KL_rate} is a specialization of the proof of \Cref{global_LS_KL_new}. Using $x_k \to x^*$ and \ref{C.3} there are $\bar k\in \mathbb{N}$ and $M >0$ such that $\|\bar B_k\| \le \|B_k\| + \cZ \le M$ and 
    \begin{equation}\label{eq:klg1}
            \varrho_{k} = \varrho(f(x_k) - f(x^*))  = c \Big[ \frac{c(1-\theta)}{\varrho^\prime(f(x_k) - f(x^*))} \Big]^{\frac{1-\theta}{\theta}} \leq c(c(1-\theta))^{\frac{1-\theta}{\theta}} \|g_k\|^{\frac{1-\theta}{\theta}}
    \end{equation}
    for all $k \ge \bar k$. Here, we applied $\varrho(t) = c t^{1-\theta}$ and $\varrho(t) = c [{c(1-\theta)}/{\varrho^\prime(t)}]^{(1-\theta)/\theta}$. 
    When $d_k = p_t$, combining \Cref{lemma: LS_SOL,lemma:stepsize} (i), we have 
    \begingroup
    \allowdisplaybreaks
    \begin{align*}
         \frac{f(x_k)-f(x_{k+1})}{\|g_k\|} &\ge \frac{\sigma \lambda_k \|g_k\| \min\{\cA , a_k \|g_k\|^{\alpha}\}}{\|\bar B_k\| + \|\bar B_k\|^2}\\
         & \ge \sigma \|g_k\|\min\Big\{s, \frac{2\rho(1-\sigma)}{L_g} \min\{\cA, a_k\|g_k\|^\alpha\}\Big\}   \frac{ \min\{\cA , a_k \|g_k\|^{\alpha}\}}{M + M^2} \\
         & \ge  \bar D \|g_k\| \min\{\cA^2, a_k^2\|g_k\|^{2\alpha}\},
       \end{align*}
    \endgroup
    where $\bar D := \sigma \min\{s,  2\rho (1-\sigma)/L_g\}/(M+M^2)$ and $L_g$ denotes the Lipschitz constant of $\nabla f$ on $\mathrm{conv}(\{x_k\})$. When $d_k =  \frac{\|g_k\|}{\|r_{t-1}\|}r_{t-1}$, by \eqref{eq:why-not?}, we obtain
    \begin{equation*}
        {[f(x_k)-f(x_{k+1})]}/{\|g_k\|} \ge 0.5{\sigma\lambda_k^2\zeta_k} \|g_k\| \ge \bar E \|g_k\| \min\{\cZ, z_k\|g_k\|^{\zeta}\}.
    \end{equation*}
    where $\bar E : = \sigma \bar\lambda^2 /2$ 
    and $\bar\lambda$ is introduced in the proof of \Cref{global_LS_KL_new}. Finally, if $d_k = -g_k$, combining \Cref{lemma: LS_GD} (i) and \eqref{eq:lipschitz}, it follows $(f(x_k)-f(x_{k+1}))/\|g_k\| \geq \sigma \min\{s,2\rho(1-\sigma)/L_g\}\|g_k\| =: \bar F \|g_k\|$. Hence, we may infer 
%
    \begin{equation}\label{eq:KL_all2}
        {[f(x_k)-f(x_{k+1})]}/{\|g_k\|}  \ge \bar G \|g_k\| \min\{1, a_k^2\|g_k\|^{2\alpha},z_k\|g_k\|^\zeta\},
    \end{equation}
     for all $k\geq \bar k$, where $\bar G := \min\{\bar D\cA^2,\bar E\min\{1,\cZ\},\bar F\}$. Next, we introduce $\mathcal{I}_1  = \{ k \ge \bar k: 1 < \min\{ a_k^2\|g_k\|^{2\alpha}, z_k\|g_k\|^\zeta\} \}$, $\mathcal{I}_2  = \{ k \ge \bar k : a_k^2\|g_k\|^{2\alpha} < \min\{1, z_k\|g_k\|^\zeta\} \}$, $\mathcal{I}_3 = \{ k \ge \bar k : k \notin \mathcal I_1 \cup \mathcal I_2 \}$, $\Gamma_k  = {\sum}_{i=k}^{\infty} \|g_i\|$, and $\Gamma_{k, \mathcal{I}_j} = {\sum}_{i = k, i \in \mathcal{I}_j}^\infty \|g_i\|$, $j=1,2,3$.
    %
    %
Mimicking \eqref{eq:use-this-one-new}--\eqref{eq:sum1-new} and using \eqref{eq:KL_all2} and the reverse H\"older inequality, we have
    \begin{equation*}
        \begin{aligned}
            {\varrho_k}/{\bar G} &\ge {\sum}_{i=k}^{\infty} {(f(x_i)- f(x_{i+1}))}/{\|g_i\|} \\
            &\ge {\sum}_{i \in \mathcal{I}_1} \|g_i\| +  {\sum}_{i \in \mathcal{I}_2} a_i^2 \|g_i\|^{1+2\alpha} +  {\sum}_{i \in \mathcal{I}_3} z_i \|g_i\|^{1+\zeta} \\
            & \ge \Gamma_{k, \mathcal{I}_1} +  \Big({\sum}_{i\geq k,i \in \mathcal{I}_2} a_i^{-2/(2\alpha)}\Big)^{-2\alpha} \Gamma_{k, \mathcal{I}_2}^{1+2\alpha} + \Big({\sum}_{i\geq k, i \in \mathcal{I}_3} z_i^{-1/\zeta}\Big)^{-\zeta}\Gamma_{k, \mathcal{I}_3}^{1+\zeta}\\
            & = \Gamma_{k, \mathcal{I}_1} + \sigma_{1,k}^{-2\alpha} \Gamma_{k, \mathcal{I}_2}^{1+2\alpha} + \sigma_{2,k}^{-\zeta} \Gamma_{k, \mathcal{I}_3}^{1+\zeta}.
        \end{aligned}
    \end{equation*}
    Here,  $\sigma_{1,k}$ and $\sigma_{2,k}$ are defined in \eqref{eq:defsigma-new}. Moreover, since $\varrho_k \to 0$ as $k \to \infty$, there exist some $k' \ge \bar k$ and some constant $\bar G^\prime>0$ such that for all $k \ge k'$, it holds that
    \begin{align} \label{eq:klg2}
        \Gamma_k 
        & \le \frac{\varrho_k}{\bar G} + \Big(\frac{\varrho_k \sigma_{1,k}^{2\alpha}}{\bar G}\Big)^{\frac{1}{1+2\alpha}} +  \Big(\frac{\varrho_k \sigma_{2,k}^\zeta}{\bar G}\Big)^{\frac{1}{1+\zeta}} \le \Big(\frac{\varrho_k}{\bar G^\prime}\Big)^{\frac{1}{1+\max\{2\alpha, \zeta\}}}.
    \end{align}
    Setting $\omega = \frac{(1+\max\{2\alpha, \zeta\})\theta}{1-\theta}$, $C_\theta = [(\bar G^\prime/c)^{\theta/(1-\theta)}/(c(1-\theta))]^{-\frac{1}{\omega}}$ and combining \eqref{eq:klg1}, \eqref{eq:klg2}, this yields
    $$ \Gamma_k^{-\omega}(\Gamma_k - \Gamma_{k+1}) \ge C_\theta^{-\omega}. $$
    This recursion allows us to derive the stated rate for $\|g_k\|$. We omit detailed computations and refer to \cite[Theorem 5.5, part (ii)]{ouyang2021trust}. (\cite{ouyang2021trust} handles the special case $\max\{2\alpha, \zeta\}=0$).
\end{proof}

\section{Proof of \texorpdfstring{\Cref{lem:auxiliary}}{Lemma 3.16}} \label{app:auxiliary}
\begin{proof}
By assumption, there is some sequence $\{w_k\}$ such that $u_k \le w_k$ and $w_k \to 0$ q-linearly. Hence, there exist some $k_1 \in \mathbb{N}$ and $\mu \in (0,1)$ such that ${w_{k+1}}/{w_{k}} \le \mu$ for $k \ge k_1$. Then due to $v_{k+1}/v_{k} \to 1$, for any $q\in (\mu,1)$, there exists $k_2 \in  \mathbb{N}$ such that ${v_{k+1}}/{v_{k}} \le q/\mu$ for $k \ge k_2$. Thus, for all $k \ge \max\{k_1, k_2\}$, we have $\frac{w_{k+1}v_{k+1}}{w_kv_k} \le q <1$ and we can infer $ \lim_{k \to \infty}  w_k v_k= 0$. Finally, by $u_k \le w_k$, this yields $\lim_{k \to \infty} u_k v_k =0$.
\end{proof}

\bibliographystyle{spmpsci}      
\bibliography{Commonbib}    


\end{document}